\documentclass{amsart}

\AtBeginDocument{%
  \providecommand\BibTeX{{%
    \normalfont B\kern-0.5em{\scshape i\kern-0.25em b}\kern-0.8em\TeX}}}
    
\usepackage{type1cm}

\usepackage{subcaption}

\newcommand{\I}{\mathbb{I}}

\newcommand{\R}{{\mathbb{R}}}

\newcommand{\N}{{\mathbb{N}}}

\newtheorem{theorem}{Theorem}[section]
\newtheorem{definition}{Definition}
\newtheorem{remark}{Remark}
\newtheorem{lemma}[theorem]{Lemma}
\newtheorem{proposition}[theorem]{Proposition}
\newtheorem{corollary}[theorem]{Corollary}
\newtheorem{example}{Example}
\newtheorem{problem}{Problem}
\newtheorem{assumption}{Assumption}

\usepackage{cite}
\usepackage{amsmath,amssymb,amsfonts}
\usepackage{algorithmic}
\usepackage{graphicx}
\usepackage{textcomp}
\usepackage{hyperref}

\usepackage{mathtools}

\usepackage{multirow}
\usepackage[bb=boondox]{mathalfa}
\newcommand{\0}{\mathbb{0}}

\newcommand{\reach}[2]{\mathcal{R}_{\Sigma} \hspace*{-0.25em} \left( #1, #2 \right)}
\newcommand{\reachU}[1]{\mathcal{R}_{\Sigma} \hspace*{-0.25em} \left( #1 \right)}

\newcommand{\proj}[1]{\pi_{n} \left( #1 \right)}

\newcommand{\TL}{\mathrm{\text{$(\tau,\lambda)$}}}

\newcommand{\mL}{\mathrm{L}}
\newcommand{\mB}{\mathcal{B}}
\newcommand{\conv}{\mathrm{conv}}

\newcommand{\C}{\mathcal{C}}
\newcommand{\Cx}{\mathcal{C}_{x}}
\newcommand{\CTL}{\mathcal{C}_{x, \TL}}
\newcommand{\CmL}{\mathcal{C}_{x, \mL}}

\newcommand{\iC}{\mathcal{C}_{xv}}
\newcommand{\iCTL}{\mathcal{C}_{xv, \TL}}
\newcommand{\iCmL}{\mathcal{C}_{xv, \mL}}

\newcommand{\blk}{\mathrm{blkdiag}}

\usepackage{makecell}

\usepackage[foot]{amsaddr}

\begin{document}

\title[Controlled invariant sets: implicit closed-form representations and applications]{Controlled invariant sets: implicit closed-form representations and applications}

\author{Tzanis Anevlavis$^{*}$}
\address{T.Anevlavis and P.Tabuada are with the
Department of Electrical and Computer Engineering\\
University of California at Los Angeles,
Los Angeles, CA 90095 USA}
\email{t.anevlavis@ucla.edu}

\author{Zexiang Liu$^{*}$}
\address{Z.Liu and N.Ozay are with the Department of Electrical Engineering and Computer Science \\ 
University of Michigan, Ann Arbor, MI 48109 USA}
\email{zexiang@umich.edu}

\author{Necmiye Ozay}
\email{necmiye@umich.edu}

\author{Paulo Tabuada}
\email{tabuada@ucla.edu}

\thanks{This work was partially supported by the CONIX Research Center, one of six centers in JUMP, a Semiconductor Research Corporation (SRC) program sponsored by DARPA, and also by NSF grants 1553873, 1918123, 1931982.}

\thanks{$^{*}$ T.Anevlavis and Z.Liu are co-first authors.}

\renewcommand{\shortauthors}{Anevlavis \emph{et al.}}

\maketitle

\begin{abstract}
    We revisit the problem of computing (robust) controlled invariant sets for discrete-time linear systems. 
    Departing from previous approaches, we consider implicit, rather than explicit, representations for controlled invariant sets. Moreover, by considering such representations in the space of states and finite input sequences we obtain \emph{closed-form} expressions for controlled invariant sets. An immediate advantage is the ability to handle high-dimensional systems since the closed-form expression is computed in a single step rather than iteratively. 
    To validate the proposed method, we present thorough case studies illustrating that in safety-critical scenarios the implicit representation suffices in place of the explicit invariant set. 
    The proposed method is complete in the absence of disturbances, and we provide a weak completeness result when disturbances are present. 
\end{abstract}

\section{Introduction}

In an increasingly autonomous world, safety has recently come under the spotlight. A safety enforcing controller is understood as one that indefinitely keeps the state of the system within a set of safe states notwithstanding the presence of uncertainties. A natural solution that guarantees safety is to initialize the state of the system inside a Robust Controlled Invariant Set (RCIS) within the set of safe states. Any RCIS is defined by the property that any trajectory starting within, can always be forced to remain therein and, hence, inside the set of safe states. Consequently, RCISs are at the core of controller synthesis for safety-critical applications. 

Since the conception of the standard method for computing the Maximal RCIS of discrete-time systems \cite{bertsekas1972infreach}, which is known to suffer from poor scaling with the system's dimension and no guarantees of termination, numerous approaches attempted to alleviate these drawbacks. A non-exhaustive overview is found in Section \ref{subsec:litreview}. 

An alternative approach is to construct an implicit representation for an RCIS. The specific implicit representation used in this paper is a set in the higher dimensional space of states and finite input sequences. We argue that in many practical, safety-critical applications, such as Model Predictive Control (MPC) and supervisory control, knowledge of the explicit RCIS is not required and the implicit representation suffices. Consequently, by exploiting the efficiency of the implicit representation the aforementioned ideas are suitable for systems with large dimensions. 

In this manuscript, we propose a general framework for \emph{computing (implicit) RCISs} for discrete-time linear systems with additive disturbances, under polytopic state, input, and even mixed, constraints. We consider RCISs parameterized by collections of \emph{eventually periodic} input sequences and prove that this choice leads to a closed-form expression for an implicit RCIS in the space of states and finite input sequences. Moreover, this choice results in a systematic way to obtain larger RCISs, which we term a \emph{hierarchy}. 
Essentially, the computed sets include all states for which there exist eventually periodic input sequences that lead to a trajectory that remains within the safe set indefinitely.
Once the (implicit) RCIS is computed, any controller rendering the RCIS invariant can be used in practice and a fixed periodic input is not chosen or used. Moreover, we show that this parameterization is rich enough, such that: 1) in the absence of disturbances, our method is complete and sufficient to approximate the Maximal CIS arbitrarily well; 2) in the presence of disturbances, a weak completeness result is established, along with a bound for the computed RCIS that can be approximated arbitrarily well. Finally we study, both theoretically and experimentally, safety-critical scenarios and establish that the efficient implicit representation suffices in place of computing the exact RCIS. In practice, the use of implicit RCISs can be done via optimization programs, e.g., a Linear Program (LP), a Mixed-Integer (MI) program, or a Quadratic Program (QP), and is only limited by the size of the program afforded to solve. 

In order to make for a more streamlined presentation, a review of the existing related literature is found at the end of the manuscript.

\textbf{Notation:} Let $\R$ be the set of real numbers and $\N$ be the set of positive integers. For sets $P, Q \subseteq \R^n$, the Minkowski sum is \mbox{$P+Q = \{x\in \R^n | x=p+q, p\in P, q\in Q\}$} and the Minkowski difference is $Q-P = \{x\in \R^n | x+p\in Q, \forall p\in P\}$.
By slightly abusing the notation, the Minkowski sum of a singleton $\{x\}$ and a set $P$ is $x+P$. 
The Hausdorff distance between $P$ and $Q$, denoted by $d(P,Q)$, is induced from the Euclidean norm in $\R^n$.
We denote a block-diagonal matrix $M$ with blocks $M_1, \dots, M_N$ by $M = \blk ( M_1, \dots, M_N )$. Moreover, given a matrix $A\in\R^{m \times n}$ and a set $P\subseteq\R^n$, the linear transformation of $P$ through $A$ is \mbox{$A P = \{A x\in \R^m | x \in P\}$}. Given a set $S \subset \R^n \times \R^m$, its projection onto the first $n$ coordinates is $\proj{S}$. 
For any $N\in\N$, let \mbox{$[N]=\{1,2, \cdots,N\}$}. 
Let $\I$ and $\mathbb{0}$ be the identity and zero matrices of appropriate sizes respectively, while $\mathbb{1}$ is a vector with all entries equal to $1$. 

\section{Problem formulation}
\label{sec:setup}
Let us begin by providing the necessary definitions.
\begin{definition}[Discrete-time linear system]
	\label{def:dtls}
	A \emph{Discrete-Time Linear System} (DTLS) $\Sigma$ is a linear difference equation: 
	\vspace{-1.1mm}
	\begin{align}
	\label{eq:dtls}
		x^+ = A x + B u + E w ,
	\end{align}
	where $x \in \R^n$ is the state of the system, $u \in \R^m$ is the input, and $w \in W \subseteq \R^{d}$ is a disturbance term. Moreover, we have that $A \in \R^{n \times n}$,  \mbox{$B \in \R^{n \times m}$}, and \mbox{$E \in \R^{n \times d}$}. 
\end{definition}

\begin{definition}[Polytope]
	\label{def:polyhedron}
	A \emph{polytope} $S \subset \R^n$ is a bounded set of the form: 
	\begin{align}
	\label{eq:polytope}
		S =  \left \{ x \in \R^n \ \middle| \ G x \leq f \right \} , 
	\end{align}
	where $G \in \R^{k \times n}$, $f \in \R^k$ for some $k>0$. 
\end{definition}

\begin{definition}[Robust Controlled Invariant Set]
	\label{def:cis}
	Given a DTLS $\Sigma$ and a safe set $S_{xu} \subset \R^n\times\R^m$, that is, the set defining the state-input constraints for $\Sigma$, a set $\mathcal{C} \subseteq \proj{S_{xu}}$ is a \emph{Robust Controlled Invariant Set} for $\Sigma$ within $S_{xu}$ if:
	\begin{align*}
		x \in \mathcal{C} \Rightarrow \exists u \in \R^m \text{ s.t. } (x,u) \in S_{xu}, A x + B u + EW \subseteq \mathcal{C} . 
	\end{align*}
\end{definition}

\begin{definition}[Admissible Input Set]
\label{def:admissible}
	Given an RCIS $\mathcal{C}$ of a DTLS $\Sigma$ within its safe set  $S_{xu}$,  the set $\mathcal{A}(x, \mathcal{C}, S_{xu})$ of admissible inputs at a state $x$ is: 
	\begin{align*}
	\mathcal{A}(x, \mathcal{C}, S_{xu}) =  \{u\in\R^m | (x,u)\in S_{xu}, Ax+Bu+EW\subseteq \mathcal{C}\}.
	\end{align*}
\end{definition}

\begin{assumption}
\label{asmpt:prbAssumption}
In this manuscript we focus on systems and safe sets that satisfy the following:\\
\hspace*{1em} 1) There exists a suitable state feedback transformation that makes the matrix $A$ of system $\Sigma$ \emph{nilpotent}. For a nilpotent matrix, there exists a $\nu \in \N$ such that $A^\nu = 0$. 	\\
\hspace*{1em} 2) The \emph{safe set} $S_{xu} \subset \R^n \times \R^m$ and the \emph{disturbance set} $W \subset \R^d$ are both polytopes.
\end{assumption}

For any system $\Sigma$ satisfying Assumption \ref{asmpt:prbAssumption}, let \mbox{$K\in \R^{m\times n}$} be the feedback gain such that $A+BK$ is nilpotent. We construct a system $\Sigma'$ by pre-feedbacking $\Sigma$ with $u=Kx+u'$: 
\begin{align*}
    x^{+} = (A+BK)x + Bu' + Ew,
\end{align*}
where $u'\in\R^m$ is the input of the system $\Sigma'$. The safe set for $\Sigma'$ is the polytope induced from the safe set $S_{xu}$ of $\Sigma$ as $S_{xu}' = \{(x,u')\in\R^n\times\R^m \mid (x, Kx + u') \in S_{xu}\}$. 

\begin{remark} 
\label{rem:equiv}
It is easy to verify that any RCIS $\C$ of $\Sigma$ within $S_{xu}$ is also an RCIS of $\Sigma'$ within $S_{xu}'$ and vice versa. Since RCISs of the two systems are the same, the problem of computing an RCIS of $\Sigma$ within $S_{xu}$ is \emph{exactly equivalent} to the problem of finding an RCIS of $\Sigma'$ within $S_{xv}'$.
\end{remark}

Given Remark \ref{rem:equiv}, the state feedback transformation simply moves the original constraints to the transformed space and, thus, it neither changes the original problem nor restricts the control authority. Although constraints on the pair $(x,u)$ become constraints on the pair $(x,u')$, they are still affine and can be handled by the proposed algorithms. Therefore, in the remainder of the paper, we assume that the system in \eqref{eq:dtls} is already pre-feedbacked, with the matrix $A$ being nilpotent.

\begin{remark}
\label{rem:nilpotentsystems}
    For any controllable system $\Sigma$ there exists a state feedback transformation satisfying Assumption~\ref{asmpt:prbAssumption} \cite[Ch.3]{antsaklis1997linear}. In this case, the nilpotency index $\nu$ is equal to the largest controllability index of $\Sigma$. 
\end{remark} 

The main goal of this paper is to compute an \emph{implicit representation of an RCIS in closed-form}. Hereafter, we refer to this representation as the \emph{implicit RCIS}. 

\begin{definition}[Implicit RCIS]
	\label{def:icis}
	Given a DTLS $\Sigma$, a safe set $S_{xu} \subset \R^n\times\R^m$, and some integer $q\in \N$, a set $\iC \subseteq \R^n\times\R^{q}$ is an \emph{Implicit RCIS} for $\Sigma$ if its projection $\proj{\iC}$ onto the first $n$ dimensions is an RCIS for $\Sigma$ within $S_{xu}$. 
\end{definition}

The following result stems directly from Definition~\ref{def:cis}.
\begin{proposition}
\label{prop:cisunionconvexhull}
	The union of RCISs and the convex hull of an RCIS are robustly controlled invariant.
\end{proposition}

For dynamical systems, i.e., systems $\Sigma$ as in \eqref{eq:dtls} where \mbox{$B=0$}, the analogous concept to RCISs is defined below.
\begin{definition}[Robust Positively Invariant Set]
	\label{def:pis}
	Given a dynamical system $\Sigma \ : \ x^+ = Ax + Ew$ and a safe set $S_{x} \subset \R^n$, a set $\mathcal{C} \subset \R^n$ is a \emph{Robust Positively Invariant Subset} (RPIS) for $\Sigma$ within $S$ if:
	\begin{align*}
		x \in \mathcal{C} \Rightarrow & \forall w \in W, \ A x + E w \in \mathcal{C} .
	\end{align*}
\end{definition}

We define the \emph{accumulated disturbance set} at time $t$ by:
\begin{align}
\label{eq:acc_dist_set}
    \overline{W}_t &=  \sum^{t}_{i=1} A^{i-1} E W.
\end{align}
By nilpotency of $A$ we have that:
\begin{align}
\label{eq:min_rpis}
\overline{W}_{\infty} &=  \sum^{ \infty}_{i=1} A^{i-1} E W = \sum^{\nu}_{i=1} A^{i-1} E W.
\end{align}
In the literature, $\overline{W}_{\infty}$ is called the \emph{Minimal RPIS} of the system $x^{+} = Ax  + Ew$ \cite{rakovic2005invariant}. 

The next operator is used throughout this manuscript. 
\begin{definition}[Reachable set]
\label{def:reach}
Given a DTLS $\Sigma$ and a set $X \subset \R^n$, define the \emph{reachable set} from $X$ under input sequence $\{u_{i}\}_{i=0}^{t-1}$ as: 
\begin{align}
\label{eq:reachableset}
    \hspace*{-0.5em}
	\reach{X}{\{u_i\}_{i=0}^{t-1}} = A^t X + \sum^t_{i=1} A^{i-1} B u_{t-i}
	+ \overline{W}_t .
\end{align}
\end{definition}
Intuitively, $\reach{X}{\{u_{i}\}_{i=0}^{t-1}}$ maps a set $X$ and an input sequence $\{u_{i}\}_{i=0}^{t-1}$ to the set of all states that can be reached from  $X$ in $t$ steps when applying said input sequence. 
Conventionally, $\reach{X}{\emptyset} = X$ and when $X$ is a singleton, i.e., $X = \{x\}$, we abuse notation to write $\reach{x}{\{u_{i}\}_{i=0}^{t-1}}$. 

\section{Implicit representation of controlled invariant sets for linear systems}
\label{sec:fin-rep-cis}
The classical algorithm that computes the \emph{Maximal} RCIS consists of an iterative procedure \cite{bertsekas1972infreach, glover1971linsysdist} and theoretically works for any discrete-time system and safe set. However, this approach is known to suffer from the curse of dimensionality and its termination is not guaranteed. To alleviate these drawbacks, we propose an algorithm that is guaranteed to terminate and computes an implicit RCIS efficiently in closed-form, thus being suitable for high dimensional systems. 
Moreover, by optionally projecting the implicit RCIS back to the original state-space one computes an explicit RCIS. Overall, the proposed algorithm computes controlled invariant sets in one and two moves respectively. 

The goal of this section is to present a \emph{finite implicit representation} of an RCIS. That is, we provide a \emph{closed-form} expression for an \emph{implicit RCIS} characterized by constraints on the state and on a finite input sequence, whose length is the design parameter. This results in a polytopic RCIS in a higher dimensional space. Intuitively, the implicit RCIS contains the pairs of states and appropriate finite input sequences that guarantee that the state remains in the safe set indefinitely.

\subsection{General implicit robust controlled invariant sets}
We begin by discussing a general construction of a polytopic implicit RCIS. First, we consider inputs $u_t$ to $\Sigma$ that evolve as the output of a linear dynamical system, $\Sigma_C$, whose state is a \emph{sequence of $q$ inputs, $v$}, i.e.: 
\begin{align}
\label{eq:linearcontroller}
    \Sigma_C \quad : \quad
	\begin{split}
	v_{t+1} = P v_{t},	\\
	u_t = H v_{t},
\end{split}
\end{align}
where $v \in \R^{m q}$, $P \in \R^{m q \times m q}$, and $H \in \R^{m \times m q}$.  
The choice of a linear dynamical system stems from our safe set being a polytope per Assumption~\ref{asmpt:prbAssumption}. By using system $\Sigma_C$ we preserve the linearity of the safe set constraints and we are, hence, able to compute polytopic RCISs within polytopic safe sets. The resulting input to $\Sigma$ can be expressed as:
\begin{align}
\label{eq:genPolicy}
	u_t = H v_t = H P^{t} v_0	,
\end{align}
for an initial choice of $v_0 \in \R^{m q}$.
We can then lift system $\Sigma$, after closing the loop with $\Sigma_C$, to the following companion dynamical system:
 \begin{align}
 \label{eq:companionsystem}
 	\Sigma_{xv} \quad : \quad
 	\begin{bmatrix} x^+ \\ v^+ \end{bmatrix}
 	= 
 	\begin{bmatrix} A & B H \\ \mathbb{0} & P \end{bmatrix}
 	\begin{bmatrix} x \\ v \end{bmatrix}
 	+
 	\begin{bmatrix} E \\ \mathbb{0} \end{bmatrix} w.
 \end{align}
 
Given the safe set $S_{xu}$, we construct the companion safe set \mbox{$S_{xv} = \left\{(x,v)\in\R^n\times\R^{mq} \mid (x,H v)\in S_{xu} \right\}$}. The companion system of \eqref{eq:dtls} is the closed-loop dynamics of \eqref{eq:dtls} with a control input in~\eqref{eq:genPolicy}. Then, the companion safe set simply constrains the closed-loop state-input pairs in the original safe set, i.e., $(x_t, H v_t)\in S_{xu}$. 

\begin{theorem}[Generalized implicit RCIS]
\label{thm:implicit_RCIS2}
Let $\iC$ be an RPIS of the companion system $\Sigma_{xv}$ within the companion safe set $S_{xv}$. The projection of $\iC$ onto the first $n$ coordinates, $\proj{\iC}$, is an RCIS of the original system $\Sigma$ within $S_{xu}$. In other words, $\iC$ is \emph{an implicit RCIS} of $\Sigma$.
\end{theorem}
\begin{proof}
  Let $x\in \proj{\iC}$. Then, there exists a $v \in \R^{mq}$ such that \mbox{$(x,v)\in \iC$}. Define $u = Hv$ and pick an arbitrary $w\in W$. By construction of $S_{xv}$, $(x,u)\in S_{xu}$. Since $\iC$ is an RPIS, we have that $(x^+, v^+)=(Ax+Bu+Ew, Pv)\in \iC$ and thus $x^+\in \proj{\iC}$. By Definition \ref{def:cis}, $\proj{\iC}$ is an RCIS of $\Sigma$ in $S_{xu}$. 
\end{proof}

In what follows, we study the conditions on $P$ and $H$ such that the Maximal RPIS of $\Sigma_{xv}$ is represented in closed-form. 

\subsection{Finite reachability constraints}
\label{sec:finreachcons}
By definition of the companion safe set $S_{xv}$ and Definition~\ref{def:pis}, we have that any state $(x,v)$ belongs to the Maximal RPIS of $\Sigma_{xv}$ within $S_{xv}$, if and only if, the input sequence $\{u_{i}\}_{i=0}^{t-1}$, with each input as in \eqref{eq:genPolicy}, satisfies:
\begin{align}
\label{eq:reachCons}
	\left( \reach{x}{\{u_{i}\}_{i=0}^{t-1}}, u_t \right) \subseteq S_{xu}, \ t \geq 0	,
\end{align}
where $\reach{x}{\{u_{i}\}_{i=0}^{t-1}} \subseteq \R^n$, $u_t \in \R^m$, and the pair $\left( \reach{x}{\{u_{i}\}_{i=0}^{t-1}}, u_t \right) \subseteq \R^n\times\R^m$. 
By Theorem \ref{thm:implicit_RCIS2}, the above constraints characterize the states and input sequences within an implicit RCIS of $\Sigma$, such that the pair $(x,u)$ stays inside the safe set $S_{xu}$ indefinitely. 
Notice that \eqref{eq:reachCons} defines an \emph{infinite} number of constraints in general. In this section, we investigate under what conditions we can reduce the above constraints into a \emph{finite} number and compute them explicitly. Then, we use these constraints to construct the promised implicit RCIS.  

\begin{definition}[Eventually periodic behavior]
\label{def:eventuperiodinput}
Consider two integers \mbox{$\tau \in \N \cup \{0\}$} and \mbox{$\lambda \in \N$}. A control input $u_t$ follows an \emph{eventually periodic behavior} if:
\begin{align}
\label{eq:eventuperiodinput}
	u_{t+\lambda} = u_{t}, ~\text{for all $t \geq \tau$}.
\end{align}
We call $\tau$ the \emph{transient} and $\lambda$ the \emph{period}.
\end{definition}

\begin{proposition}[Finite reachability constraints] 
\label{prop:finitereachcons}
	Consider a DTLS $\Sigma$ satisfying Assumption~\ref{asmpt:prbAssumption}. If the input $u_t$ follows an eventually periodic behavior with transient $\tau \in \N\cup\{0\}$ and period $\lambda \in \N$, then the infinite constraints in \eqref{eq:reachCons} are reduced to a finite number of constraints.
\end{proposition}
\begin{proof}
    Under Assumption 1 the matrix $A$ is nilpotent with nilpotency index $\nu$. Consequently, given \eqref{eq:reachableset}, the reachable set from a state $x$ for $t \geq \nu$ depends only on the past $\nu$ inputs. We abuse notation to write $\reachU{\{u_{i}\}_{i=0}^{t-1}}$ and omit the state $x$ to denote dependency only on the inputs. Then, for $t \geq \nu+\tau$: 
	\begin{align*}
        &\reachU{\{u_{i}\}_{i=0}^{t-1}} = \sum^{\nu}_{i=1} A^{i-1} B u_{t-i} + \overline{W}_{\infty} \\ 
		&\overset{\text{\eqref{eq:eventuperiodinput}}}= \sum^{\nu}_{i=1} A^{i-1} B u_{t+\lambda-i} + \overline{W}_{\infty} 
		= \reachU{\{u_{i}\}_{i=0}^{t+\lambda-1}}.
	\end{align*}
	Therefore, under inputs with eventually periodic behavior the reachability constraints repeat themselves after $t = \nu+\tau+\lambda$. As a result, we can split the constraints in \eqref{eq:reachCons} as:
\begin{align}
	\label{eq:Preach1}
	&\left(\reach{x}{\left\{u_i\right\}_{i = 0}^{t-1}}, u_t\right) \subseteq S_{xu}, ~ t = 0, \dots, \nu-1 ,	\\
	\label{eq:Preach3}
	&\left(\reachU{\left\{u_i\right\}_{i = 0}^{t-1}}, u_t\right) \subseteq S_{xu}, ~ t = \nu, \dots, \nu+\tau+\lambda-1.
\end{align}
The above suggests that $\left(\reach{x}{\{u_{i}\}_{i=0}^{t-1}}, u_t\right) \subseteq S_{xu}$ for all $t\geq0$ can be replaced with only $\nu+\tau+\lambda$ constraints.
\end{proof}

Proposition~\ref{prop:finitereachcons} provides a finite representation of the constraints in \eqref{eq:reachCons} under the eventually periodic input behavior in \eqref{eq:eventuperiodinput}. The next question we address concerns characterizing the classes of policies that guarantee the behavior in \eqref{eq:eventuperiodinput}. 

\subsection{Implicit robust controlled invariant sets in closed-form}
\label{sec:iRCISclosedform}
Recall that our goal is to derive a closed-form expression for an implicit RCIS of $\Sigma$, which is essentially the Maximal RPIS of the companion system $\Sigma_{xv}$ by Theorem~\ref{thm:implicit_RCIS2}. So far we proved that, in general, inputs with eventually periodic behavior result in finite reachability constraints. Clearly, the parameterized input in \eqref{eq:genPolicy} follows an eventually periodic behavior as in \eqref{eq:eventuperiodinput} if: 
\begin{align}
\label{eq:PmatProperty}
	P^{t} = P^{t+\lambda}, ~ t \geq \tau, 
\end{align}
i.e., $P$ is an \emph{eventually periodic} matrix with transient $\tau$ and period $\lambda$. 

\begin{proposition}[Structure of eventually periodic matrices]
\label{prop:PmatProperty}
	Any \emph{eventually periodic} matrix $P \in \R^{n \times n}$ has eigenvalues that are either $0$ or $\lambda$-th roots of unity. If $\tau \neq 0$, i.e., $P$ is not purely periodic, then $P$ has at least one $0$ eigenvalue with algebraic multiplicity equal to $\tau$ and geometric multiplicity equal to $1$. If $P^\tau \neq 0$, i.e., $P$ is not nilpotent, then $P$ has at least one eigenvalue that is a $\lambda$-th root of unity. 
\end{proposition}

\begin{proof}
Let $\mathrm{v} \neq 0$ be an eigenvector of $P$ and $\delta$ the corresponding eigenvalue, i.e., $P \mathrm{v} = \delta \mathrm{v}$. Then, \eqref{eq:PmatProperty} for $t \geq \tau$ yields:
\begin{align*}
	P^{t} = P^{t+\lambda} &\Rightarrow P^{t} \mathrm{v} = P^{t+\lambda} \mathrm{v} \Leftrightarrow \delta^{t} \mathrm{v} = \delta^{t+\lambda} \mathrm{v} \\	
		&\overset{\mathrm{v}\neq0}\Leftrightarrow  \delta^{t} = \delta^{t+\lambda} \Leftrightarrow  \delta^{t} \left(1 - \delta^{\lambda} \right) = 0	,
\end{align*}
that is, the eigenvalues $\delta$ of $P$ are only $0$ or $\lambda$-th roots of unity. 

Consider now the Jordan normal form $P = M J M^{-1}$ \cite{laub2004matrixanalysis}. This form is unique up to the order of the Jordan blocks, and $P^t = M J^t M^{-1}$. Without loss of generality, we write:
\begin{align*}
J = \begin{bmatrix} J_1 & \mathbb{0} \\ \mathbb{0} & J_2 \end{bmatrix}	,
\end{align*}
where $J_1$ is the Jordan block corresponding to the eigenvalues of $P$ that are $0$, and $J_2$ is the Jordan block corresponding to the eigenvalues of $P$ that are the $\lambda$-th roots of unity. Thus, $J_1$ is nilpotent. Then, when $\tau \neq 0$, 
equality \eqref{eq:PmatProperty} is equivalent to:
\begin{align*}
	P^{t} = P^{t+\lambda} \Leftrightarrow M J^t M^{-1} = M J^{t+\lambda} M^{-1}, ~ t \geq \tau . 
\end{align*}
Matrix $J_1$ vanishes in exactly $\tau$ steps, i.e., \mbox{$J_1^\tau = 0$} and \mbox{$J_1^t \neq 0$}, for $t < \tau$. This implies that $P$ has at least one $0$ eigenvalue with algebraic multiplicity equal to $\tau$ and geometric multiplicity equal to $1$, but no $0$ eigenvalues of geometric multiplicity $1$ and algebraic multiplicity greater than $\tau$.

Moreover, when $P$ is not nilpotent, i.e., $P^\tau \neq \mathbb{0}$, for $t\geq\tau$:
\begin{align*}
	J^t = J^{t+\lambda} 
	&\overset{J_1^t=\0, t\geq\tau}\Leftrightarrow \begin{bmatrix} \0 & \0 \\ \0 & J_2^t \end{bmatrix} = \begin{bmatrix} \0 & \0 \\ \0 & J_2^{t+\lambda} \end{bmatrix}	\\
	&\Leftrightarrow J_2^t = J_2^{t+\lambda}.
\end{align*}
Thus, $P$ has at least one eigenvalue that is a $\lambda$-th root of unity.
\end{proof}

\begin{corollary}
\label{cor:PmatStruct}
The class of matrices described by Proposition \ref{prop:PmatProperty} that satisfies \eqref{eq:PmatProperty} can be written, up to a similarity transformation, in the following form: 
\begin{align}
\label{eq:PmatSol}
	P = \begin{bmatrix} N & Q \\ \mathbb{0} & R \end{bmatrix},
\end{align}
where $N$ is a nilpotent matrix with nilpotency index $\tau$, $R$ is a matrix whose eigenvalues are all $\lambda$-th roots of unity, i.e., $R^\lambda = \I$, and $Q$ is an arbitrary matrix. 
\end{corollary}

Proposition~\ref{prop:PmatProperty} and Corollary~\ref{cor:PmatStruct} guide the designer to effortlessly select matrix $P$ via its eigenvalues or its submatrices.
Moreover, it is reasonable to select the projection matrix $H$ to be \emph{surjective} in order to obtain a non-trivial input in \eqref{eq:genPolicy}. 

We now show that we can compute the desired \emph{closed-form} expression for an implicit RCIS parameterized by collections of \emph{eventually periodic} input sequences. 

\begin{theorem}[Closed-form implicit RCIS]
\label{thm:implicitRCIS}
	Consider a DTLS $\Sigma$ and a safe set $S_{xu}$ for which Assumption~\ref{asmpt:prbAssumption} holds. Select an eventually periodic matrix $P \in \R^{mq \times mq}$ and a surjective projection matrix $H \in \R^{m \times mq}$. An \emph{implicit RCIS} for $\Sigma$ within $S_{xu}$, denoted by $\iC$, is defined by the constraints:
	\begin{align}
	\label{eq:iRCISconditions}
	\begin{split}
		\left( A^t x + \sum^t_{i=1} A^{i-1} B H P^{t-i} v, H P^{t} v \right) &\subseteq S_{xu} - \overline{W}_{t} 
		\times \{0\}, ~  t=0, \dots, \nu-1	,	\\
		\left( \sum^{\nu}_{i=1} A^{i-1} B H P^{t-i} v, H P^{t} v \right) &\subseteq S_{xu} - \overline{W}_{\infty}
		\times \{0\},  ~ t = \nu, \dots, \nu+\tau+\lambda-1	 .
	\end{split}
	\end{align}
	That is, the set $\iC \subset \R^n\times\R^{m q}$:
	\begin{align}
	\label{eq:closedformRCIS}
	&\iC = \left\{ (x, v) \in \R^n \times \R^{m q} \mid  (x,v) \text{ satisfy \eqref{eq:iRCISconditions}} \right\},
	\end{align}
	is computed in closed-form. Moreover, $\iC$ is the Maximal RPIS of the companion dynamical system in~\eqref{eq:companionsystem}.
\end{theorem}

\begin{proof}
By Proposition~\ref{prop:finitereachcons}, the set $\iC$ defined by  \eqref{eq:iRCISconditions} in closed-form satisfies the constraints in \eqref{eq:reachCons} and, thus, is the Maximal RPIS of the companion system $\Sigma_{xv}$ in $S_{xv}$. Then, by Theorem~\ref{thm:implicit_RCIS2}, $C_{xv}$ is an implicit RCIS of $\Sigma$ in $S_{xu}$. 
\end{proof}

Theorem~\ref{thm:implicitRCIS} provides an implicit RCIS, $\iC$, in closed-form. This set defines pairs of states and finite input sequences such that the state remains in the safe set indefinitely.

\begin{remark}[On the choice of input behavior]
    Notice that the open-loop eventually periodic policy used to parameterize the implicit RCIS is only a means towards its computation in closed-form. In practice, after computing an RCIS, we can use any controller appropriate for the task at hand. This is illustrated in our case studies in Section~\ref{sec:simulations}, where the controller of the system is independent of the RCIS implicit representation. For instance, once an RCIS is available one defines a closed-loop non-periodic and memoryless controller \mbox{$K:\R^n \to \R^m$} for which $Ax+B K(x)$ belongs to the RCIS when $x$ is an element of the RCIS. 
\end{remark}

\begin{corollary}[Computation of explicit RCIS]
\label{thm:explicitRCIS}
By selecting an eventually periodic matrix $P \in \R^{m q \times m q}$ and a projection matrix $H \in \R^{m \times m q}$, one computes an \emph{explicit} RCIS $C_x=\proj{\iC}$ with a single projection step. 
\end{corollary}
The size of the lifted space leads to a trade-off: on the one hand it can result to larger RCISs, as we detail in the next section, but on the other it requires more effort if the optional projection step is taken. 

\section{A hierarchy of controlled invariant sets}
\label{sec:hierarchy}

Our main result, Theorem~\ref{thm:implicitRCIS}, provides a closed-form expression for an implicit RCIS, $\iC$, with constraints on the state of the system, $x$, and on a finite sequence of inputs, $v$. The resulting sets depend on the choice of the eventually periodic matrix $P$ in~\eqref{eq:linearcontroller} and the projection matrix $H$. 

In this section, we show how to \emph{systematically} construct a sequence of RCISs that form a \emph{hierarchy}, i.e., a non-decreasing sequence. Our goal is to provide a closed-form expression for the implicit RCISs corresponding to this hierarchy. Towards this, we identify special forms of matrices $P$ and $H$. 
\begin{definition}[$\TL$-lasso sequence]
\label{def:lassobehavior}
Consider two integers $\tau \in \N \cup \{0\}$ and $\lambda \in \N$, and let $q = \tau + \lambda$. 
The control input $u$ generated by the dynamical system $\Sigma_C$ in \eqref{eq:linearcontroller} forms a \emph{$(\tau,\lambda)$-lasso sequence} with respect to the inputs $v$, if: 
\begin{align}
\label{eq:lassobehavior}
	\begin{aligned} 
    P = P_{\TL} = \blk \left( \bar{P}, \dots, \bar{P} \right) \in \R^{m q \times m q}, \\ 
	H = H_{\TL} = \blk \left( \bar{H}, \dots, \bar{H} \right) \in \R^{m \times m q}, 
	\end{aligned}
\end{align}
with $m$ blocks each and $\bar{P}$, $\bar{H}$ defined as:
\begin{align}
\label{eq:CPmatrices}
\begin{aligned} 
    \bar{P} &= 
	\begin{bmatrix}
		\0 && \I &&	\\
		0 & \cdots &1& \cdots & 0	
	\end{bmatrix}
	\in \R^{q \times q}
	,
	\\
	\bar{H} &= 
	\
	\begin{bmatrix}
		1 & 0 & \ \dots & 0
	\end{bmatrix}
	\in \R^{1 \times q}
	.
\end{aligned}
\end{align}
In the last row of $\bar{P}$ the $1$ occurs at the $\tau$-th position. It is easy to verify that $P_{\TL}$ in \eqref{eq:lassobehavior} is of the form \eqref{eq:PmatSol}. A $(\tau,\lambda)$-lasso sequence has a transient of $\tau$ inputs followed by periodic inputs with period $\lambda$. 
\end{definition}

We utilize the $\TL$-lasso sequence to formalize a hierarchy of RCISs with a single decision parameter $q$. 
\begin{definition}[Lassos of same length]
\label{def:lassosL}
Select $q \in \N$. Define the set of all pairs $\TL \in \N\cup\{0\} \times \N$ corresponding to lassos of length $q$ as:
\begin{align}
\label{eq:lassosL}
	\Theta_q = \left\{ \TL \in \N\cup\{0\} \times \N \mid \tau+\lambda = q \right\}.
\end{align}
The cardinality of $\Theta_q$ is exactly $q$. 
\end{definition}

The next result provides a way to systematically construct implicit RCISs in closed-form such that the corresponding explicit RCISs form a hierarchy. 
\begin{theorem}[Hierarchy of RCISs]
\label{thm:hierarchy}
	Consider a DTLS $\Sigma$ and a safe set $S_{xu}$ for which Assumption \ref{asmpt:prbAssumption} holds, and select an integer $q \in \N$. Given $q$, the set $\iCmL \subset \R^n \times \R^{m q}$:
	\begin{align}
    \label{eq:implicitRCISL-union}
    	\iCmL &= \bigcup_{\TL \in \Theta_q} \iCTL	,
    \end{align}
	is the implicit RCIS induced by the $q$-level of the hierarchy, where each $\iCTL$ is computed in closed-form in \eqref{eq:closedformRCIS} with $P$ and $H$ as in \eqref{eq:lassobehavior}. In addition, the explicit RCIS:
    \begin{align}
        \label{eq:explicitRCISL-union}
    	\CmL = \proj{\iCmL} = \bigcup_{\TL \in \Theta_q} \proj{\iCTL} = \bigcup_{\TL \in \Theta_q} \CTL	,
    \end{align}
	corresponding to the $q$-level of the hierarchy contains any RCIS lower in the hierarchy, i.e.:
	\begin{align}
	\label{eq:containment}
		\CmL \supseteq \C_{x,\mL'},~\text{for any $q, q' \in \N$ with $q' < q$.}
	\end{align}
\end{theorem}
\begin{proof}
First, the sets $\iCmL$ and $\CmL$ are implicit and explicit RCISs respectively as the unions of, implicit and explicit, RCISs by Proposition~\ref{prop:cisunionconvexhull}. Next we prove \eqref{eq:containment} for the case of $q$ and $q+1$, while the more general statement follows by a simple induction argument. 
	
For any $\lambda \in \N$ such that $\TL \in \Theta_q$, we have by \eqref{eq:lassosL} that $(\tau+1,\lambda) \in \Theta_{q+1}$. It is easy to conceptualize that:
\begin{align}
\label{eq:projContain}
	\mathcal{C}_{x,(\tau+1, \lambda)} \supseteq \CTL, 
\end{align}
as a $\TL$-lasso sequence can always be embedded in a \mbox{$(\tau+1,\lambda)$}-lasso sequence. 
From \eqref{eq:explicitRCISL-union} it now follows that: 
\begin{align*}
	\mathcal{C}_{x,(\mL+1)} &= \bigcup_{\TL \in \Theta_{q+1}} \CTL 
	=\left( \bigcup_{\TL \in \Theta_q} \mathcal{C}_{x,(\tau+1,\lambda)} \right) \bigcup \mathcal{C}_{x,(0,q+1)}	\\
	&\overset{\eqref{eq:projContain}}\supseteq \left( \bigcup_{\TL \in \Theta_q} \CTL \right) \bigcup \mathcal{C}_{x,(0,q+1)} \overset{\eqref{eq:explicitRCISL-union}}= \CmL \bigcup \mathcal{C}_{x,(0,q+1)}.
\end{align*}
The above entails that $\mathcal{C}_{x,(\mL+1)} \supseteq \CmL$. 
\end{proof}

\begin{corollary}
\label{cor:liftedhierarchy}
    Using the standard big-M formulation, the implicit RCIS $\iCmL$ can be expressed as a projection of a higher-dimensional polytope: 
    \begin{align}
    \label{eq:liftunion}
        \C_{xv\zeta,\mL} = \left\{ (x,v,\zeta) 
        \middle| \sum_{i=1}^q \zeta_i = 1, G_i (x,v) \leq f_i + (1-\zeta_i) M \mathbb{1} \right\},
    \end{align}
    where \mbox{$\zeta \in \{0,1\}^q$}, $G_i$ and $f_i$ describe each of the $q$ polytopes $\iCTL$ in~\eqref{eq:implicitRCISL-union}, and $M\in\R_{+}$ is sufficiently large. 
    The set $\C_{xv\zeta,\mL}$ is a polytope in $\R^n\times\R^{m q}\times\{0,1\}^q$, and its projection on $\R^n\times\R^{m q}$ is exactly the union 
in~\eqref{eq:implicitRCISL-union}, while its projection on $\R^n$ is exactly 
the explicit RCIS in~\eqref{eq:explicitRCISL-union}. 
\end{corollary}

Theorem~\ref{thm:hierarchy} defines the promised hierarchy and provides an implicit RCIS for each level of the hierarchy that can also be computed in closed-form in~\eqref{eq:liftunion} at the cost of an additional lift. Fig.~\ref{fig:hierarchy_motivation} illustrates the relation in \eqref{eq:containment}, that is, how the sets induced by each hierarchy level contain the ones induced by lower hierarchy levels. 

\begin{remark}[Convex hierarchy]
We can replace the union in~\eqref{eq:implicitRCISL-union} by the convex hull \mbox{$\conv\left( \bigcup_{\TL \in \Theta_q} \iCTL \right)$}. Then, in an analogous manner, all the above results go through resulting in a \emph{hierarchy of convex RCISs}. Similarly to \eqref{eq:liftunion}, by standard set-lifting techniques, one obtains a closed-form expression for the convex hull. 
\end{remark}

\begin{remark}[Partial hierarchies without union]
\label{rem:clarification}
    The proposed hierarchy involves handling a union of sets. However, one might prefer to avoid unions of sets and rather use a single convex set. As each implicit RCIS $\iCTL$ involved in the hierarchy is computed in closed-form by Theorem~\ref{thm:implicitRCIS}, we provide two more refined guidelines for obtaining larger RCISs, based on the choice of $\TL$: 
    \begin{enumerate}
        \item Given any $\lambda\in\N$, it holds that $\C_{x,(\tau+1,\lambda)} \supseteq \CTL$ for any \mbox{$\tau \in \N\cup\{0\}$}. 
        \item Given any $\tau \in \N\cup\{0\}$, it holds that $\CTL \supseteq \C_{x,(\tau,\lambda')}$ for any $\lambda, \lambda' \in \N$ such that $\lambda = k \lambda'$, $k\in\N$, i.e., $\lambda$ is a multiple of $\lambda'$, see \cite[Section 4.6]{anevlavis2020simple} when $\tau=0$.
    \end{enumerate}
    The above can direct the designer in search of larger RCISs that are based on a single implicit RCIS. 
\end{remark}

\begin{figure}[t]
	\includegraphics[width=0.75\textwidth]{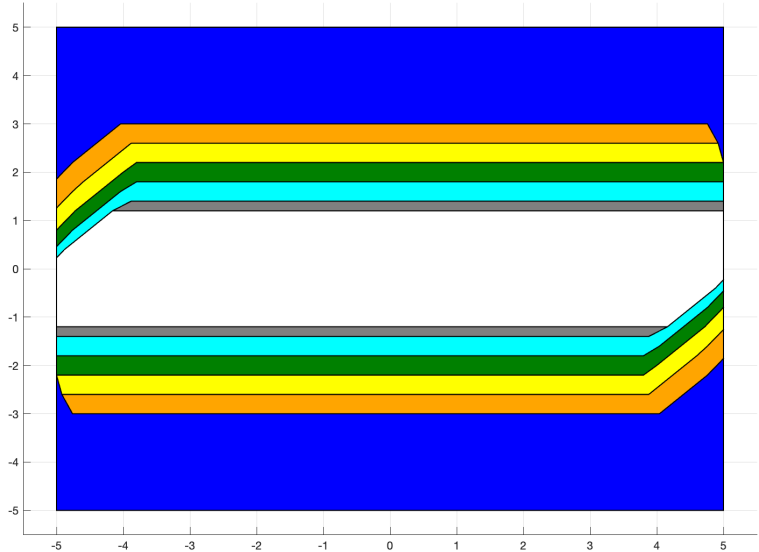}
	\caption{RCIS corresponding to $q=1$ (white), $q=2$ (gray), $q=3$ (teal), $q=4$ (green), $q=5$ (yellow), and $q=6$ (orange) for a double integrator. Safe set in blue.}
	\label{fig:hierarchy_motivation}
\end{figure}

\section{Implicit invariant sets in practice: controlled invariant sets in one move}
\label{sec:onemove}
Using the proposed results, one has the option to project the implicit RCIS back to the original space and obtain an explicit RCIS as proposed in the two-move approach \cite{anevlavis2019cis2m,anevlavis2020simple,anevlavis2021enhanced}. However, the required projection from a higher dimensional space becomes the bottleneck of this approach.

One of the goals of this manuscript is to establish that in a number of key control problems explicit knowledge of the RCIS is not required and the implicit RCIS suffices. We show how the proposed methodology can be used \emph{online} as the implicit RCIS which admits a closed-form expression. 

\subsection{Extraction of admissible inputs}
For many applications in this section, we need to extract a set of admissible inputs of the RCIS $\pi_{n}(C_{xv})$ at a given state $x$, i.e, $\mathcal{A}(x, \pi_{n}(C_{xv}), S_{xu})$ as given in Definition~\ref{def:admissible}. Given only the implicit RCIS $C_{xv}$, we provide here three linear encodings of $\mathcal{A}(x, \pi_{n}(\mathcal{C}_{xv}), S_{xu})$ or its nonempty subsets.

1) The first linear encoding of $\mathcal{A}(x, \pi_{n}(\mathcal{C}_{xv}), S_{xu})$ is given by the polytope: 
\begin{align}
    \mathcal{U}_{1}(x) = \big\{ &(u,v_{1:N})\in \R^{(1+Nq)m} \,\big\vert\, (x,u)\in S_{xu}, (Ax+Bu+Ew_{i}, v_{i})\in \mathcal{C}_{xv}, \forall i\in [N] \big\},
\end{align}
where $v_{{1:N}}$ denotes the vector $(v_{1}, v_{2}, \cdots, v_{N})$. It follows that \mbox{$\pi_{m}(\mathcal{U}_{1}(x)) = \mathcal{A}(x,\pi_{n}(C_{xv}), S_{xu})$}. 

2) The second linear encoding is: 
\begin{align}
   \mathcal{U}_{2}(x) = \big\{v\in \R^{qm} \,\big\vert\, (x,v)\in \mathcal{C}_{xv}\big\},
\end{align}
with $H$ and $P$ as in \eqref{eq:linearcontroller}. Note that $\mathcal{U}_{2}(x)$ is the slice of $\mathcal{C}_{xv}$ at $x$ and is nonempty for $x\in \pi_{n}(\mathcal{C}_{xv})$. Then, the linear transformation $H\mathcal{U}_{2}(x)$ is a nonempty subset of $\mathcal{A}(x, \pi_{n}(\mathcal{C}_{xv}), S_{xu})$. 

3) Finally, define the polytope: 
\begin{align}
    \mathcal{U}_{3}(x) = \big\{ &(u,v)\in \R^{(1+q)m} \,\big\vert\, (x,u)\in S_{xu}, (Ax+Bu+Ew_{i}, v)\in \mathcal{C}_{xv}, \forall i\in [N] \big\},
\end{align}
where $w_{i}\in \mathcal{V}$ with $\mathcal{V}$ the vertices of $W$. It follows that $\pi_{m}(\mathcal{U}_{3}(x)) \subseteq \mathcal{A}(x,\pi_{n}(\mathcal{C}_{xv}), S_{xu})$. It is easy to check that $(Hv,Pv)\in \mathcal{U}_{3}(x)$ for all $v\in \mathcal{U}_{2}(x)$, which implies that $\mathcal{U}_{3}(x)$ is guaranteed to be nonempty for any $x\in \pi_{n}(\mathcal{C}_{xv})$. 

All three linear encodings are easily computed online given $\mathcal{C}_{xv}$. Moreover, it holds that: 
\begin{align*}
   H\mathcal{U}_{2}(x) \subseteq \pi_{m}(\mathcal{U}_{3}(x)) \subseteq \pi_{m}(\mathcal{U}_{1})(x) = \mathcal{A}(x,\pi_n(\mathcal{C}_{xv}), S_{xu}). 
\end{align*}
That is, $\mathcal{U}_{2}(x)$ is the most conservative encoding, while $\mathcal{U}_{1}(x)$ is the least conservative one. However, $\mathcal{U}_{2}$ is of lower dimension, while $\mathcal{U}_{1}$ has the highest dimension. More conservative encodings are easier to compute. Depending on the available compute, a user can select the most appropriate encoding. 

\subsection{Supervision of a nominal controller}
\label{subsec:supervision}
In many scenarios, when synthesizing a controller for a plant, the objective is to meet a performance criterion while always satisfying a safety requirement. This gives rise to the problem of \emph{supervision}. 

\begin{problem}[Supervisory Control]
\label{prb:supervision}
    Consider a system $\Sigma$, a safe set $S_{xu}$, and a nominal controller that meets a performance objective. 
    The supervisory control problem asks at each time step to evaluate if, 
    given the current state, the input $\tilde{u}$ from the nominal controller keeps the next state of $\Sigma$ in the safe set. 
    If not, \emph{correct} $\tilde{u}$ by selecting an input that does so. 
\end{problem}

To solve Problem~\ref{prb:supervision} one has to guarantee \emph{at every step} that the pairs of states and inputs respect the safe set $S_{xu}$. A natural way to do so is by using an RCIS. The supervision framework operates as follows. Given an RCIS $\C$, assume that the initial state of $\Sigma$ lies in $\C$. The nominal controller provides an input $\tilde{u}$ to be executed by $\Sigma$. If $\tilde{u} \in \mathcal{A}(x, \C, S _{xu})$, then its execution is allowed. Else $\tilde{u}$ is corrected by selecting an input $u_{safe} \in \mathcal{A}(x, \C, S_{xu})$. Existence of $u_{safe}$ is guaranteed in any RCIS by Definition~\ref{def:cis}. 

In practice an explicit RCIS is not needed. One can exploit the three linear encodings of admissible inputs from the proposed implicit RCISs to perform supervision.  Furthermore, the nominal controller can be designed independently of the implicit RCIS parameterization.
Consider an implicit RCIS $\iC$ for $\Sigma$ within $S_{xu}$, as in Theorem~\ref{thm:implicitRCIS}. Then supervision of an input $\tilde{u}$ is performed by solving the following QP:
    \begin{align}
    \label{eq:supervisionQP}
        \begin{aligned}
		&\min_{u,v}  &||u - \tilde{u} ||_2^2  \hspace{42.5mm}  \\
        &\text{s.t.}    & (x,u)\in S_{xu} \hspace{37.5mm}  \\
		& &  (Ax+Bu+Ew, v)\in \mathcal{C}_{xv}, \forall w\in W \hspace{2mm} 
        \end{aligned}
    \end{align}
    Notice that the feasible domain of the QP in \eqref{eq:supervisionQP} is equal to the third linear encoding $\mathcal{U}_{3}(x)$ of admissible inputs; similar QPs are easily formulated with the feasible domain being $\mathcal{U}_{1}(x)$ or $\mathcal{U}_{2}(x)$. By solving optimization problem \eqref{eq:supervisionQP} we compute the \emph{minimally intrusive safe input}. 

\subsection{Safe online planning}
\label{subsec:planning}
Based on the discussed supervision framework, we utilize the proposed implicit RCIS to enforce safety constraints in online planning tasks. The goal here is to navigate a robot through unknown environments without collision with any obstacles. The map is initially unknown, and it is built and updated online based on sensor measurements, such as LiDAR. The robot must only operate in the detected obstacle-free region. To ensure this, given a path planning algorithm and a tracking controller, we supervise the controller inputs based on the implicit RCIS. The overall framework is shown in Figure~\ref{fig:control_framework}. 

\begin{figure}[t!]
	\centering
	\includegraphics[width=0.35\textwidth]{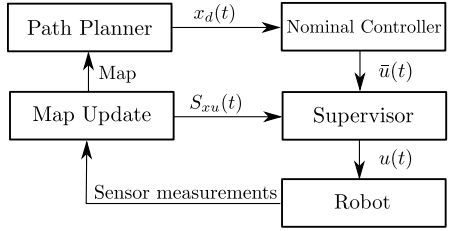}
	\caption{The overall safe online planning framework.}
	\label{fig:control_framework}
\end{figure}
The safe set for the robot imposes bounds on states and inputs, which do not change over time, and also constraints, e.g., on the robot's position, which are given by the obstacle-free region in the current map. As the detected obstacle-free region expands over time, the corresponding part of the safe set does as well. Thus, differently from Section~\ref{subsec:supervision}, we have a time-varying safe set $S_{xu}(t)$ satisfying \mbox{$S_{xu}(t) \subseteq S_{xu}(t+1)$}, $t\geq 0$. As the implicit RCIS is constructed in closed-form, we can generate a new implicit RCIS $\C_{xv}(t)$ for each $S_{xu}(t)$. Then, at each time step $t$, for any $t' \leq t$, we supervise the nominal input $\tilde{u}(t)$ by solving the optimization problem:
\begin{align*}
        \mathcal{P}(t,t') :
        \begin{aligned}
        \min_{u,v}  ~~&|| u - \tilde{u} ||_2^2  \hspace{19mm} \\
        \text{s.t.} ~~      &(x,u)\in S_{xu}(t)\\
        &(Ax+Bu+Ew, v) \in \iC(t'), \forall w\in W.	
        \end{aligned}
\end{align*}
As $S_{xu}(t) \subseteq S_{xu}(t+1)$, $\C_{xv}(t')$ is a valid implicit RCIS in $S_{xu}(t)$ for all $t \geq t'$. Thus, as long as $\mathcal{P}(t,t')$ is feasible, the optimizer $v^*$ of $\mathcal{P}(t,t')$ is a safe input that guarantees the next state lies in the RCIS. Furthermore, if $P(t,t')$ is feasible, by definition of RCIS, $P(t+1,t')$ is also feasible. Thus, if $P(0,0)$ is feasible, for all $t > 0$, there exists $t' \leq t$ such that $P(t,t')$ is feasible. That is, the recursive feasibility of $P(t,t')$ is guaranteed. 
In practice, to take advantage of the latest map, we always select $t'$ to be the latest time instant $t^*$ for which $P(t,t^*)$ is feasible. 

To summarize, at each time step, we first construct the implicit RCIS $\C_{xv}(t)$ based on the current map. Then, given the state and nominal control input, we solve $\mathcal{P}(t,t^*)$ to obtain the minimally intrusive safe input. This input guarantees that the state of the robot stays within $S_{xu}(t)$ for all $t\geq 0$, provided that $\mathcal{P}(0,0)$ is feasible.

\subsection{Safe hyper-boxes}
\label{subsec:safeboxes}
For high dimensional systems, the exact representation of an RCIS $\Cx$ can be a set of thousands of linear inequalities. This reduces insight as it is quite difficult to clearly identify regions of each state that lie within the RCIS. In contrast, hyper-boxes are easy to grasp in any dimension and immediately provide information about the regions of states they contain. Based on this, we explore how implicit RCISs can be used to find hyper-boxes that can be considered \emph{safe} in the following sense. 

\begin{definition}[Safe hyper-boxes]
\label{def:safebox}
Consider a system $\Sigma$, a safe set $S_x$, and the Maximal RCIS $\C_{max} \subseteq S_x$. Define a hyper-box \mbox{$\mB = [\underline{b}_1, \overline{b}_1] \times \dots \times  [\underline{b}_n, \overline{b}_n] =  [\underline{b}, \overline{b}] \subset \R^n$}. We call a hyper-box $\mB$ \emph{safe} if $\mB \subseteq \C_{max}$. 
\end{definition}

To simplify the presentation we only consider state constraints, $S_x$, instead of $S_{xu}$. Notice that by Definition~\ref{def:safebox}, a safe hyper-box is not necessarily invariant. A safe hyper-box $\mB$ entails the guarantee that the trajectory starting therein can remain in $S_x$ forever, but not necessarily within $\mB$. We now aim to address the following problem. 

\begin{problem}
\label{pr:largestsafebox}
	Find the largest{\footnote{The largest, as measured by volume, hyper-box within a set might not be unique. We choose a heuristic for maximizing the volume of a set that yields a well-defined convex optimization problem. Hence, the term ``largest'' refers to the heuristic used.}} \emph{safe} hyper-box $\mB$ within $\Cx$. 
\end{problem}

A hyper-box $\mB$ can be described by a pair of vectors \mbox{$\left(\underline{b},\overline{b}\right) \in \R^{n} \times \R^n$}. Then, using similar arguments to Section~\ref{sec:fin-rep-cis}, we compute in closed-form an implicit RCIS $\C_\mB$ characterizing all hyper-boxes \mbox{$\left(\underline{b},\overline{b}\right)$} that remain in $S_x$ under eventually periodic inputs. The eventually periodic inputs are given by a vector $v \in \R^{m q}$ with $q = \tau+\lambda$. Then, the set $\C_\mB$ lives in \mbox{$\R^{n} \times \R^n \times \R^{m q}$} and is described by:
\begin{small}
\begin{align*}
\label{eq:closedformCB}
	\begin{split}
	&A^t \left[ \underline{b},\overline{b} \right] + \sum^t_{i=1} A^{i-1} B H P^{t-i} v \subseteq S_x - \overline{W}_{t}, t=0, \dots, \nu-1	,	\\
	&\sum^{\nu}_{i=1} A^{i-1} B H P^{t-i} v \subseteq S_x - \overline{W}_{\infty}, t = \nu, \dots, \nu+q-1	 .
	\end{split}
\end{align*}
\end{small}
The above constraints can all be written as linear inequalities in \mbox{$\left(\underline{b},\overline{b}, v \right) \in \R^{n} \times \R^n \times \R^{m q}$}. Then, the implicit RCIS $\C_\mB$ is a polytope and one solves Problem~\ref{pr:largestsafebox} by the following convex optimization program:
\begin{align*}
\begin{split}
    \max_{\left(\underline{b}, \overline{b}, v \right)} 	&\gamma \left( \overline{b} - \underline{b} \right)	\\
    \text{s.t. } 	\hspace*{2.75mm}&	\left(\underline{b}, \overline{b}, v \right) \in \mathcal{C}_{\mB}, 
\end{split}
\end{align*}
where $\gamma(y) = \left( \Pi_{i = 1}^n y_i \right)^{\frac{1}{n}}$ is the geometric mean function, which is used as a heuristic for the volume of the hyper-box. Function $\gamma$ is concave, and maximizing a concave function can be cast as a convex minimization problem \cite{boyd2004convex}. 

\begin{remark}[Invariant and recurrent hyper-boxes]
\label{rem:invar-recurr-boxes}
Two special cases of the above are \emph{invariant} hyper-boxes, when $\tau=0$, $\lambda=1$, see also \cite{anevlavis2019cis2m}, and \emph{recurrent} hyper-boxes, when $\tau=0$, $\lambda>0$, see also \cite{anevlavis2020simple, anevlavis2021enhanced}.
\end{remark}

A related question to Problem~\ref{pr:largestsafebox} is to evaluate if a proposed hyper-box is safe. This is of interest when evaluating whether the initial condition of a problem or an area around a configuration point $x_c$ where the system is required to operate is safe. If both the above are modeled by hyper-boxes $(\underline{b}, \overline{b})$, we can simply ask whether there exists a $v$, such that $(\underline{b}, \overline{b}, v) \in \mathcal{C}_{\mB}$. Similarly, more complicated questions can be formulated, e.g., to find the largest safe box around a configuration point. 

\begin{remark}[Complexity when using implicit RCISs] 
\label{rem:implicitComplexity}
In this section we showed how several key problems in control are solved without the need of projection and of an explicit RCIS, which results in extremely efficient computations since the implicit RCISs are computed in closed-form. The decision to be made is the size of the lift, i.e., the length of the input sequence. From a computational standpoint, this choice is only limited by how large an optimization problem one affords solving given the application. 
\end{remark}

\section{Performance bound for the proposed method}
\label{sec:perfbound}
Numerical examples, to be presented later, will show that the projection of the proposed implicit RCIS onto the original state-space can coincide with the Maximal RCIS. However, this is not always the case. When there is a gap between our projected set and the Maximal RCIS, one may wonder if that gap is fundamental to our method. In other words, can we arbitrarily approximate the Maximal RCIS with the projection of our implicit RCIS by choosing better $P$ and $H$ matrices? 

In this section we aim to answer the above question and provide insights into the completeness of our method. Given \eqref{eq:min_rpis}, define the nominal DTLS $\overline{\Sigma}$ and the nominal safe set $\overline{S}_{xu}$:
\begin{align}
    &\overline{\Sigma} \ : \ x^{+} = Ax + Bu,    \\
    &\overline{S}_{xu}  = S_{xu} - \overline{W}_{\infty}\times \{0\},
\end{align}
where $A$ and $B$ are the same as in \eqref{eq:dtls}.  Let $\overline{\C}_{\max}$ be the Maximal CIS of the nominal system $ \overline{\Sigma}$ within $ \overline{S}_{xu}$ and define: 
\begin{align} 
\label{eq:C_nominal_pre} 
\begin{aligned}
	\C_{outer,\nu} = \Big\{x \in \R^{n} \,\Big\vert\, &\exists \{u_{i}\}_{i=0}^{\nu-1} \in \R^{m\nu}, \left(\reach{x}{\{u_{i}\}_{i=0}^{t-1}},u_{t}\right) \hspace*{-1mm}\subseteq\hspace*{-1mm} S_{xu}, t = 0,\dots, \nu-1,   \\
	&\reach{x}{\{u_{i}\}_{i=0}^{\nu-1}} \subseteq \overline{\C}_{\max} + \overline{W}_{\infty} \Big\},
\end{aligned}
\end{align}
where $\nu$ is the nilpotency index of $A$.

\begin{proposition}
   $\C_{outer,\nu}$ is an RCIS of $\Sigma$ within $S_{xu}$.
\end{proposition}
\begin{proof}
In this proof, we use the order cancellation lemma, as a special case of \cite[Thm. 4]{grzybowski2019order}.
\begin{lemma} \label{thm:inv_sum} 
	Let $X, Y \subset \R^n$ be two closed convex sets with $Y$ bounded. A point $x\in \R^{n}$ is in $X $ if and only if $x+Y \subseteq X+Y$. 
\end{lemma}
To prove that $\C_{outer,\nu}$ is an RCIS, we show that for any \mbox{$x_0\in \C_{outer,\nu}$}, there exists $u$ such that $(x_0,u)\in S_{xu}$ and for all $w\in W$, $Ax_0+Bu+Ew \in \C_{outer,\nu}$. By definition of $\C_{outer,\nu}$, there exists a sequence $\{u_{i}\}_{i=0}^{\nu-1}$ that, along with $x_0$, satisfies the conditions in~\eqref{eq:C_nominal_pre}. We aim to show that $u_0$ in $\{u_{i}\}_{i=0}^{\nu-1}$ is a feasible choice for $u$. Given~\eqref{eq:C_nominal_pre}, the reachable set from $x_0$ at time $\nu$ is:
\begin{align*} 
\reach{x_0}{\left\{ u_{i} \right\}_{i=0}^{\nu-1}} = \sum^{\nu-1}_{i=0}A^{\nu-1-i} Bu_{i}+\overline{W}_{\infty} \subseteq \overline{\C}_{\max} + \overline{W}_{\infty},
\end{align*}
with $\overline{W}_{\infty}$ and $\overline{\C}_{\max}$ convex and $\overline{W}_{\infty}$ bounded. By Lemma~\ref{thm:inv_sum} we have that $\sum^{\nu-1}_{i=0}A^{\nu-1-i} Bu_{i} \in \overline{\C}_{\max}$. Since $\overline{\C}_{\max}$ is controlled invariant within $\overline{S}_{xu}$ for the nominal DTLS $\overline{\Sigma}$, there exists $u_{\nu}$ such that:
\begin{small}
\begin{align}
	&\left(\sum^{\nu-1}_{i=0} A^{\nu-1-i}Bu_{i}, u_{\nu} \right)\in  \overline{S}_{xu}, \label{eq:pf_RCIS_0} \\ 
	&A\left( \sum^{\nu-1}_{i=0} A^{\nu-1-i}Bu_{i} \right)+Bu_{\nu} = \sum^{\nu}_{i=1} A^{\nu-1-i}Bu_{i}\in \overline{\C}_{\max}. \nonumber 
\end{align}
\end{small}
Consider any $w \in W$ and define $x_1 = Ax + Bu_0 + Ew$: 
\begin{small}
\begin{align}
\label{eq:pf_RCIS_1} 
\mathcal{R}_{\Sigma}(x_1,\{u_{i}\}_{i=1}^{\nu}) =\sum^{\nu}_{i=1} A^{\nu-1-i}Bu_{i}+ \overline{W}_{\infty} \subseteq\overline{\C}_{\max} + \overline{W}_{\infty}. 
\end{align}
\end{small}
From \eqref{eq:pf_RCIS_0} we have that:
\begin{align} \label{eq:pf_RCIS_3} 
	(\mathcal{R}_{\Sigma}(x_1, \{u_{i}\}_{i=1}^{\nu-1}),u_{\nu})  \subseteq S_{xu}.
\end{align}
Finally, note that for $t=0, \cdots, \nu-2$, we have:
\begin{align} \label{eq:pf_RCIS_2} 
	\left(\mathcal{R}_{\Sigma}(x_1, \{u_{i}\}_{i=1}^{t}),u_{t+1}\right) \subseteq \left(\mathcal{R}_{\Sigma}(x_0, \{u_{i}\}_{i=0}^{t}),u_{t+1}\right) \subseteq S_{xu}.
\end{align}
From \eqref{eq:pf_RCIS_1}, \eqref{eq:pf_RCIS_3}, and \eqref{eq:pf_RCIS_2} we verify that $x_1\in \C_{outer,\nu}.$ Thus, $\C_{outer,\nu}$ is an RCIS.
\end{proof}

The following theorem shows that $\C_{outer,\nu}$ is an outer bound of the projection of the proposed implicit RCIS.
\begin{theorem}[Outer bound on $\proj{\iC}$] 
\label{thm:outer} 
    For a companion system $\Sigma_{xv}$ as in \eqref{eq:companionsystem}, with arbitrary matrices $P$ and $H$, let $\iC$ be an RPIS of $\Sigma_{xv}$ within the companion safe set $S_{xv}$. The RCIS $\proj{\iC}$ is a subset of $\C_{outer,\nu}$, that is \mbox{$\proj{\iC} \subseteq \C_{outer,\nu}$}. 
\end{theorem} 
\begin{proof}
	Let $x\in \proj{\iC}$. We show that $x\in \C_{outer,\nu}$. By definition of $\iC$, there exists a vector $v$ such that:
	\begin{align} \label{eq:R_HP} 
		\left(\mathcal{R}_{\Sigma}\left(x, \left\{ H P^{i}v \right\}_{i=0}^{t-1}\right),H P^{t}v\right) \subseteq S_{xu}, \text{ for all $t \geq 0$}.
	\end{align}
	Define $u_{t} = H P^{t}v$. We want to verify that $x$ and $\{u_{i}\}_{i=0}^{\nu - 1}$ satisfy the two conditions in the definition of~\eqref{eq:C_nominal_pre}. The first condition is immediately satisfied by~\eqref{eq:R_HP}. It is left to show that $ \mathcal{R}_{\Sigma}(x, \{u_{i}\}_{i=0}^{\nu-1}) \subseteq \overline{\C}_{\max}+ \overline{W}_{\infty}$. That is:
	\begin{align*}
	    \sum^{\nu-1}_{i=0} A^{\nu-1-i}Bu_{i} +\overline{W}_{\infty} \subseteq \overline{\C}_{\max}+ \overline{W}_{\infty}. 
	\end{align*}
	By Lemma \ref{thm:inv_sum}, it is equivalent to prove that: 
	\begin{align*}
	   \overline{x}\equiv \sum^{\nu-1}_{i=0} A^{\nu-1-i}Bu_{i} \in\overline{\C}_{\max}.
	\end{align*}
By \eqref{eq:R_HP}, we have that for $t \geq 0$:
\begin{align}
    \begin{aligned}
	&\left( \sum^{\nu-1}_{i=0}A^{\nu-1-i} Bu_{i+t} + \overline{W}_{\infty} ,u_{v+t}\right) \subseteq S_{xu}
	\Leftrightarrow \left( \sum^{\nu-1}_{i=0}A^{\nu-1-i} Bu_{i+t}  ,u_{\nu+t}\right) \in \overline{S}_{xu}\\
		\Leftrightarrow & \left( \mathcal{R}_{\overline{\Sigma}}(\overline{x}, \{u_{i}\}_{i=\nu}^{\nu+t-1}),u_{\nu+t}  \right) \in \overline{S}_{xu}
    \end{aligned}
    \label{eq:x_bar_inv} 	
\end{align}
According to \eqref{eq:x_bar_inv}, the control sequence $ \{u_{i}\}_{i=\nu}^{\nu+t-1}$ guarantees that the trajectory of $ \overline{\Sigma}$ starting at $ \overline{x}$ stays within $ \overline{S}_{xu}$ for all $t \geq 0$. Thus, $ \overline{x}$ must belong to the Maximal CIS of $\Sigma$ in $ \overline{S}_{xu}$. That is, $\overline{x}\in \overline{\C}_{\max}$.
\end{proof}

Note here that the set $\C_{outer,\nu}$, which serves as an outer bound for the set computed by our method, is as hard to compute as the Maximal RCIS. Given Theorem~\ref{thm:outer} we have:
\begin{align}
\label{eq:cis_order}
	\proj{\iC} \subseteq \C_{outer,\nu} \subseteq \C_{\max}.
\end{align}
Thus, the projection of our implicit RCIS can coincide with the Maximal RCIS, for appropriately selected matrices $P$ and $H$, only if $\C_{outer,\nu} = \C_{\max}$ in \eqref{eq:cis_order}. This potential gap between our approximation and the Maximal RCIS is due to the fact that our method uses open-loop forward reachability constraints under disturbances. Finally, the following theorem establishes weak completeness of our method.
\begin{theorem}[Weak completeness]
\label{thm:complete}
    The set $\C_{outer,\nu}$ is nonempty, if and only if, there exist matrices $P$ and $H$ such that the corresponding implicit RCIS $\iC$ is nonempty. Specifically, $\C_{outer,\nu}\neq\emptyset$, if and only if, $\C_{xv,(0,1)}\neq\emptyset$, that is $P$ and $H$ are as in \eqref{eq:lassobehavior} with $\TL=(0,1)$. 
\end{theorem}
\begin{proof}
	We want to show that $\C_{outer,\nu}$ is nonempty if and only if $\C_{xv,(0,1)}$ is nonempty, where $\C_{xv,(0,1)}$ is defined in \eqref{eq:implicitRCISL-union} with respect to system $\Sigma$ and safe set $S_{xu}$.

	Since $\proj{\C_{xv,(0,1)}} \subseteq \C_{outer,\nu}$, immediately nonemptyness of $\C_{xv,(0,1)}$ implies nonemptyness of $\C_{outer,\nu}$. 

	For the converse, suppose that $\C_{outer,\nu}$ is nonempty. Then $\overline{\C}_{\max}$ is nonempty. By \cite[Theorem 12]{caravani2002doubly}, we know that $\overline{\C}_{\max}$ is nonempty, if and only if, there exists a fixed point $x \in \overline{\C}_{\max}$ along with a $u$ such that $(x,u) \in \overline{S}_{xu}$ and $Ax+Bu = x$. Also, note that $A\overline{W}_{\infty} + EW = \overline{W}_{\infty}$. Thus, we have: 
	\begin{align} \label{eq:W_inv} 
		\begin{aligned}
		&(x+\overline{W}_{\infty},u) \subseteq S_{xu},\\
		&A(x+ \overline{W}_{\infty}) + Bu + EW = x + \overline{W}_{\infty}. 
		\end{aligned}
	\end{align}
	According to \eqref{eq:W_inv}, for any $y\in x+ \overline{W}_{\infty}$, we have $(y,u)\in S_{xu}$ and $Ay+Bu+EW \subseteq x+\overline{W}_{\infty}$, which implies that $x+ \overline{W}_{\infty}$ is an RCIS of $\Sigma$ within $S_{xu}$. By the definition of $\C_{xv,(0,1)}$, it is easy to check that $(x+\overline{W}_{\infty},u) \subseteq \C_{xv,(0,1)}$. Thus, $\C_{xv,(0,1)}$ is nonempty.
\end{proof}
\begin{corollary}[Completeness in absence of disturbances]
    In the absence of disturbances, $\C_{outer,\nu} = \C_{max}$ and, thus, there exist $P$ and $H$ such that $\iC$ is nonempty, if and only if, $\C_{max}$ is nonempty. That is, the proposed method is complete. 
\end{corollary}

The significance of Theorem~\ref{thm:complete} lies in allowing to quickly check nonemptiness of $\C_{outer,\nu}$ by computing $\C_{xv,(0,1)}$, which we can do in closed-form. 
Even though the gap between $\C_{outer,\nu}$ and $\C_{\max}$ is still an open question at the writing of this manuscript, we show that $\proj{\iC}$ can actually converge to its outer bound for a specific choice of $H$ and $P$ matrices. 

\begin{theorem}[Convergence to $\C_{outer,\nu}$]
\label{thm:convergence}
	Assume that the disturbance set $W$ contains $0$, and the interior of $\overline{S}_{xu}$ contains a fixed point $(x,u)$ of $ \overline{\Sigma}$. There exist matrices $H$ and $P$ such that $\proj{\iC}$ approaches $\C_{outer,\nu}$. Specifically, if $H$ and $P$ are as in \eqref{eq:lassobehavior}, by increasing $\tau$ in \eqref{eq:lassobehavior}, $\proj{\iC}$ converges to $\C_{outer,\nu}$ in Hausdorff distance exponentially fast. 
\end{theorem}
\begin{proof}
Without loss of generality, assume that the fixed point $(x,u)$ of $\overline{\Sigma}$ in the interior of $\overline{S}_{xu}$ is the origin of the state-input space. We define a set operator $ \mathcal{U}(\C)$ that maps a subset $\C$ of $\R^{n}$ to a subset of $\R^{\nu m}$:
\begin{align}
\label{eq:u_set_op}
   \mathcal{U}(\C) = \bigg\{u_{0:\nu-1}\in \R^{\nu m} \,\bigg\vert\, \sum_{i=1}^{\nu} A^{i-1}Bu_{\nu-i}\in \C\bigg\},
\end{align}
where $u_{0:\nu-1}$ denotes the vector $(u_0,u_1, \cdots, u_{\nu-1})\in \R^{\nu m}$.

To maintain a streamlined presentation, we make the following claims that we prove in Appendix~\ref{sec:proof_of_claims}. 

\emph{Claim 1}: The polytope $\C_{xv,0}$ contains the origin, where: 
\begin{align*}
		\C_{xv,0} =  \{(x, u_{0:\nu-1})\in \R^{n+\nu m} \mid \left(\reach{x}{\{u_{i}\}_{i=0}^{t-1}},u_{t}\right) \hspace*{-1mm}\subseteq\hspace*{-1mm} S_{xu}, t = 0,\dots, \nu-1	\}. \nonumber
\end{align*}

\emph{Claim 2}: For the set $\C_{outer, \nu}$ in \eqref{eq:C_nominal_pre} it holds that: 
	\begin{align} 
	\label{eq:43} 
	    \C_{outer, \nu} = \pi_{n} (\C_{xv,max}),
	\end{align}
where $\C_{xv,max} = \C_{xv,0}\cap ( \R^{n} \times \mathcal{U}(\overline{\C}_{max}))$.

\emph{Claim 3}: 
Let $\overline{\C}_{xv, (\tau, \lambda)}$ be the implicit CIS of the nominal system  $\overline{\Sigma}$ within $\overline{S}_{xu}$ with $H$ and $P$ as in \eqref{eq:lassobehavior} and let \mbox{$\overline{\C}_{x, (\tau, \lambda)} = \pi_{n}(\overline{\C}_{xv, (\tau, \lambda)})$}. 
The implicit RCIS $\C_{xv, (\tau, \lambda)}$ of $\Sigma$ within $S_{xu}$ with $H$ and $P$ as in \eqref{eq:lassobehavior} satisfies: 
\begin{align} 
\label{eq:45} 
\pi_{n}(\C_{xv, (\tau, \lambda)})= \pi_{n}(\widehat{\C}_{xv,(\tau, \lambda)}), \text{ for any $\tau \geq \nu$}, 
\end{align}
where $\widehat{\C}_{xv, (\tau, \lambda)} = \C_{xv,0} \cap (\R^{n}\times\mathcal{U}(\overline{\C}_{x, (\tau-\nu, \lambda)}))$. 

\emph{Claim 4}: 
There exist $c_0 > 0$, $a\in [0,1)$, and some $\tau_1 \geq 0$ such that for any $ \lambda \geq 1$ and for any $\tau \geq \tau_1$: 
\begin{align} 
\label{eq:46} 
\overline{\C}_{x, (\tau, \lambda)} \supseteq (1- c_0a^{\tau})\overline{\C}_{max}, 
\end{align}
with $\tau_1$ big enough such that $1-c_0a^{\tau_1} \geq 0$ and thereby the right hand side of \eqref{eq:46} is well-defined.

We use these claims to prove the desired convergence rate. The operator $\mathcal{U}(\cdot)$ in \eqref{eq:u_set_op} is linear with respect to scalar multiplication, i.e, $\mathcal{U}( \xi \mathcal{C}) = \xi \mathcal{U}( \mathcal{C})$, $ \xi \geq 0$, and monotonic, i.e., $\mathcal{U}(\C_1) \supseteq \mathcal{U}(\C_2)$, $\C_1 \supseteq \C_2$. According to \eqref{eq:46}, for $\tau \geq \tau_1$: 
\begin{align}
  \mathcal{U}(\overline{\C}_{x, (\tau, \lambda)}) \supseteq (1-c_0a^{\tau})\mathcal{U}(\overline{\C}_{max}). 
\end{align}
Note $\tau_0 = \nu + \tau_1$. By \eqref{eq:45}, for $\tau \geq \tau_0$:
\begin{align}
	\begin{split} \label{eq:48} 
		\widehat{\C}_{xv,(\tau, \lambda)}	&	\supseteq  \C_{xv,0}\cap(1-c_0a^{\tau-\nu}) (\R^{n}\times \mathcal{U}(\overline{\C}_{max}))
		\supseteq  (1-c_0a^{\tau-\nu})(\C_{xv,0}\cap (\R^{n}\times \mathcal{U}(\overline{\C}_{max}))) \\
	&	\supseteq  (1-c_0a^{\tau-\nu})\C_{xv,max}.
	\end{split}
\end{align}
The second inclusion above holds since $0\in \C_{xv,0}$ and thus $(1-c_0a^{\tau})\C_{xv,0} \subseteq \C_{xv,0}$.
Note that $\pi_{n}(\cdot)$ is also linear with respect to scalar multiplication. By \eqref{eq:43}, \eqref{eq:45} and \eqref{eq:48}, for $\tau \geq \tau_0$:
\begin{align}
	\label{eq:49} 
		\C_{x,(\tau, \lambda)}= \pi_{n}(\C_{xv,(\tau, \lambda)}) = \pi_{n}(\widehat{\C}_{xv, (\tau, \lambda)}) 
	\supseteq \pi_{n}((1-c_0 a^{\tau-\nu})\C_{xv,max}) = (1-c_0a^{\tau-\nu}) \C_{outer,\nu}. 
\end{align}
By Theorem \ref{thm:outer} and \eqref{eq:49}, for any $\tau \geq  \tau_0$: 
\begin{align} \label{eq:51} 
 (1-c_0a^{\tau-\nu}) \C_{outer,\nu} \subseteq \C_{x,(\tau, \lambda)} \subseteq \C_{outer,\nu}. 
\end{align}
Let $c_1= \max_ {x_1,x_2\in \C_{outer,\nu}}\Vert x_1-x_2\Vert_{2}$ be the diameter of $\C_{outer,\nu}$, which is finite since $S_{xu}$ is bounded. Then, by \eqref{eq:51}, the Hausdorff distance between $\C_{x,(\tau, \lambda)}$ and $\C_{outer,\nu}$ satisfies: 
\begin{align*}
    d(\C_{x,(\tau, \lambda)}, \C_{outer,\nu})\leq c a^{\tau},~\text{for $c= c_0 c_1a^{-\nu}$ and $\tau \geq \tau_0$}.
\end{align*}
\end{proof}

Note that $\C_{outer,\nu}$ contains the union of the projections  $\pi_n(\C_{xv})$ for all general implicit RCISs $\C_{xv}$ suggested by Theorem \ref{thm:implicit_RCIS2} (that is, the matrices $H$ and $P$ can be arbitrary, not necessarily the eventually periodic ones in Section \ref{sec:iRCISclosedform}). Hence, intuitively the set $\C_{outer,\nu}$ should be much larger than the projection of any specific implicit RCIS $\C_{xv}$ corresponding to an eventually periodic $H$ and $P$ in Section \ref{sec:iRCISclosedform}. However, Theorem \ref{thm:convergence} shows that the proposed implicit RCIS can approximate $\C_{outer,\nu}$ arbitrarily well by just using the simple $H$ and $P$ matrices as in \eqref{eq:CPmatrices}. Moreover, the approximation error decays exponentially fast as we increase the parameter $\tau$ in \eqref{eq:CPmatrices}. This result implies that the eventually periodic input structure explored in Section III.B and III.C is rich enough, and not as conservative as what it may look at first sight. 

\begin{corollary}
\label{cor:nominal_converge}
   In the absence of disturbances, if the interior of $S_{xu}$ contains a fixed point of $\Sigma$, then for any \mbox{$\lambda>0$}, then $\mathcal{C}_{x,(\tau, \lambda)}$ converges to the Maximal CIS  $\C_{max}$ in Hausdorff distance exponentially fast as $\tau$ increases. 
\end{corollary}

The condition that the interior of $S_{xu}$ (resp. $\overline{S}_{xu}$) contains a fixed point of $\Sigma$ (resp. $\overline{\Sigma}$) in Corollary \ref{cor:nominal_converge} (resp. Theorem \ref{thm:convergence}) is critical to our method:  
\begin{example}
   Let $\Sigma$ be $x_1^+=x_2$, $x_2^+=u$ and the safe set $S_{xu} = \{(x,u) | -1 \leq x_1 $, $ 1.5x_2 \leq x_1 \leq  2 x_2, u\in [-1,1]\}$. The only fixed point of $\Sigma$ in $S_{xu}$ is the origin in $\R^{3}$, which is also a vertex of $S_{xu}$. It is easy to check that $\C_{max} = \pi_{n}(S_{xu})$, but the largest CIS $\pi_{n}(C_{xv})$ computed by our method is equal to the singleton set $\{0\}$. 

If we expand $S_{xu}$ slightly so that its interior contains the origin, there immediately exist $H$ and $P$ such that $\pi_{n}(\iC)$ approximates $\C_{max}$ arbitrarily well, as expected by Corollary~\ref{cor:nominal_converge}. Conversely, if we slightly shrink $S_{xu}$ so that it does not contain any fixed point, then $\C_{max}$ is empty\cite[Theorem12]{caravani2002doubly}. 
\end{example}

\begin{remark}
Under the assumption that $0\in W$, let $\mathbb{S}_{xu}$ be the set of all the polytopic safe sets $S_{xu}$ that have a nonempty $\C_{outer,\nu}$. Moreover, let $\partial \mathbb{S}_{xu}$ be the set of all safe sets $S_{xu}\in \mathbb{S}_{xu}$, whose corresponding nominal safe set $\overline{S}_{xu}$ does not contain a fixed point of $\overline{\Sigma}$ in the interior. It can be shown that $\partial \mathbb{S}_{xu}$ must be contained by the boundary of $\mathbb{S}_{xu}$ in the topology induced by Hausdorff distance. Consequently, for any safe set in the interior of $\mathbb{S}_{xu}$, there exists $H$ and $P$ such that $\pi_n(\iC)$ approximates $\C_{outer,\nu}$ arbitrarily well. 
\end{remark}

\section{Case studies}
\label{sec:simulations}
A MATLAB implementation of the proposed method, along with instructions to replicate our case studies, can be found at \url{https://github.com/janis10/cis2m}. 

\subsection{Quadrotor obstacle avoidance using explicit RCIS}
We begin by tackling the supervision problem, defined in Section~\ref{subsec:supervision}, for the task of quadrotor obstacle avoidance. That is, we filter nominal inputs to the quadrotor to ensure collision-free trajectories. The dynamics of the quadrotor can be modeled as a non-linear system with 12 states \cite{mueller2013mpcquadcopter}. Nonetheless, this system is differentially flat, which implies that the states and inputs can be rewritten as a function of the so-called flat outputs and a finite number of their derivatives \cite{zhou2014flatquadrotors}. Exploiting this property, we obtain an equivalent linear system that expresses the motion of a quadrotor. Moreover, the original state and input constraints can be overconstrained by polytopes in the flat output space \cite{pannocchi2021iros}. Then, the motion of a quadrotor can be described by:
\begin{align*}
   x^+ = A x + B u + E w,
\end{align*}
with \mbox{$A = \blk(A_1,A_2,A_3)$}, \mbox{$B=\blk(B_1,B_2,B_3)$}, and:
\begin{align*}
A_i = 
\begin{bmatrix}
	1 & T_s & \frac{T_s^2}{2!}	\\
	0 & 1 & T_s	\\
	0 & 0 & 1
\end{bmatrix}
,
~
B_i = 
\begin{bmatrix}
	\frac{T_s^3}{3!} \\
	\frac{T_s^2}{2!} \\
	T_s
\end{bmatrix}
.
\end{align*}
The state $x\in\R^9$ contains the 3-dimensional position, velocity, and acceleration, while the input $u \in \R^3$ is the 3-dimensional jerk. The matrix $E$ and disturbance $w$ are selected appropriately to account for various errors during the experiment. 

The operating space for the quadrotor is a hyper-box with obstacles in $\R^3$, see Fig.~\ref{fig:coolplot}. The safe set is described as the obstacle-free space, a union of overlapping hyper-boxes in $\R^3$, along with box constraints on the velocity and the acceleration: 
\begin{align*}
    S = \bigcup_{j=1}^N \left[ \underline{p}_j,\overline{p}_j \right] \times \left[ \underline{v},\overline{v} \right] \times \left[ \underline{a},\overline{a} \right] ,
\end{align*}
where $[ \underline{p}_j,\overline{p}_j ] \subset \R^3$, for $j=1,\dots,N$, is a hyper-box in the obstacle-free space, $\left[ \underline{v},\overline{v} \right] \subset \R^3$ and $\left[ \underline{a},\overline{a} \right] \subset \R^3$ denote the velocity and acceleration constraints respectively. The safe set is a union of polytopes, while our framework is designed for convex polytopes. Since we already know the obstacle layout, we compute an explicit RCIS for each polytope in the safe set. As these polytopes overlap we expect, and it is actually the case in our experiments, that the RCISs do so as well. This allows, when performing supervision, to select the input that keeps the quadrotor into the RCIS of our choice when in the intersection of overlaping RCISs and, hence, navigate safely. 

\begin{figure}[t]
	\includegraphics[width=0.70\textwidth]{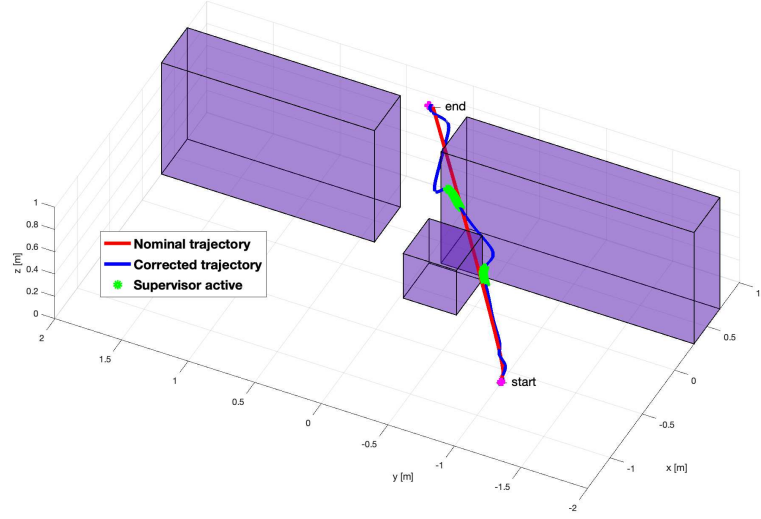}
	\vspace{-10mm}
	\caption{Quadrotor operational region. Obstacles in purple transparent boxes. 
	Nominal trajectory (red), corrected trajectory (blue), supervision active (green).}
	\label{fig:coolplot}
\end{figure}

\begin{figure}[t!]
\centering
\begin{subfigure}{.5\textwidth}
	\includegraphics[width=1.0\linewidth]{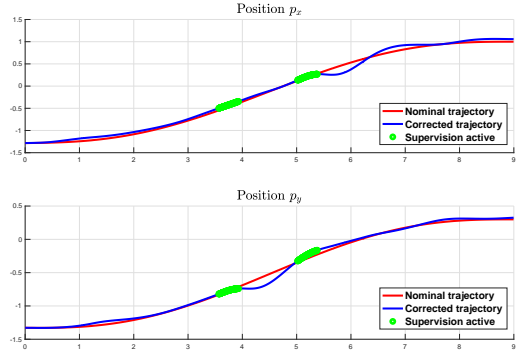}
	\caption{Quadrotor position}
	\label{fig:position}
\end{subfigure}
\begin{subfigure}{.5\textwidth}
	\includegraphics[width=1.0\linewidth]{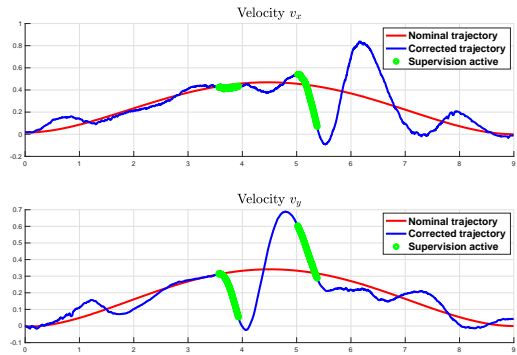}
    \caption{Quadrotor velocity}
	\label{fig:velocity}
\end{subfigure}
\begin{subfigure}{.5\textwidth}
	\includegraphics[width=1.0\linewidth]{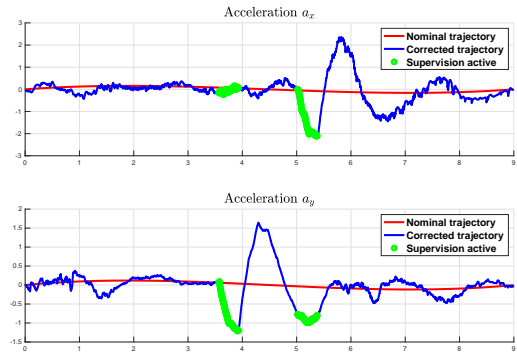}
    \caption{Quadrotor acceleration}
	\label{fig:acceleration}
\end{subfigure}
\caption{Quadrotor trajectory in $x$-$y$ plane: nominal trajectory (red), corrected trajectory (blue), supervision active (green).}
\label{fig:trajectory}
\end{figure}

Our goal is to ensure collision-free trajectory tracking. In Fig.~\ref{fig:coolplot}, the nominal trajectory (red line) moves the quadrotor from a start point to an end point through the obstacles. As we can appreciate, the supervised trajectory (blue curve) takes the quadrotor around the obstacles and, safely, to the end point. When the supervision is active, the quadrotor performs more aggressive maneuvers to avoid the obstacle as shown in Fig.~\ref{fig:velocity} and Fig.~\ref{fig:acceleration}, where we omit the \mbox{$z$-axis} as in this experiment the quadrotor maintains a relatively constant altitude. A video of the experiment is found at \url{https://tinyurl.com/drone-supervision-cis}. For visualizing the trajectory and the obstacles in the video, we used the Augmented Reality Edge Networking Architecture (ARENA) \cite{arena}.

In this experiment we utilized the explicit RCIS \mbox{$\CTL=\proj{\iCTL}$} with $\TL=(0,6)$ and the one-step projection was done in just several seconds for this specific system. Our hardware platform is the open-source Crazyflie 2.0 quadrotor. The operating space for the position is $[-2,2]\times[-2,2]\times[0,1]$ (measured in $m$) and the obstacles are shown in Fig.~\ref{fig:coolplot}. The velocity, acceleration, and jerk constraints are \mbox{$[-1.0, 1.0]$} (measured in $m/s$), \mbox{$[-2.83,2.83]$} (in $m/s^2$), and \mbox{$[-59.3,59.3]$} (in $m/s^3$) respectively. The sampling time is $T_s = 0.18 s$. For the state estimation we use a Kalman filter, where the measurements are the quadrotor's position and attitude as obtained by the motion capture system OptiTrack. The nominal controller is a feedback controller stabilizing the error dynamics between the current state and a tracking point in the nominal trajectory. The optimization problems were solved by GUROBI \cite{gurobi}. 

\subsection{Safe online planning using implicit RCIS}
\label{subsec:planning-exp}
Next, we solve the safe online planning problem, discussed in Section~\ref{subsec:planning}, for ground robot navigation. The map is initially unknown and is built online based on LiDAR measurements. While navigating the robot needs to avoid the obstacles, indicated by the dark area in Fig.~\ref{fig:map}, and reach the target point. This case study is inspired by the robot navigation problem in \cite{bajcsy2019efficient}.

\begin{figure}[t]
	\centering
	\includegraphics[width=0.5\textwidth]{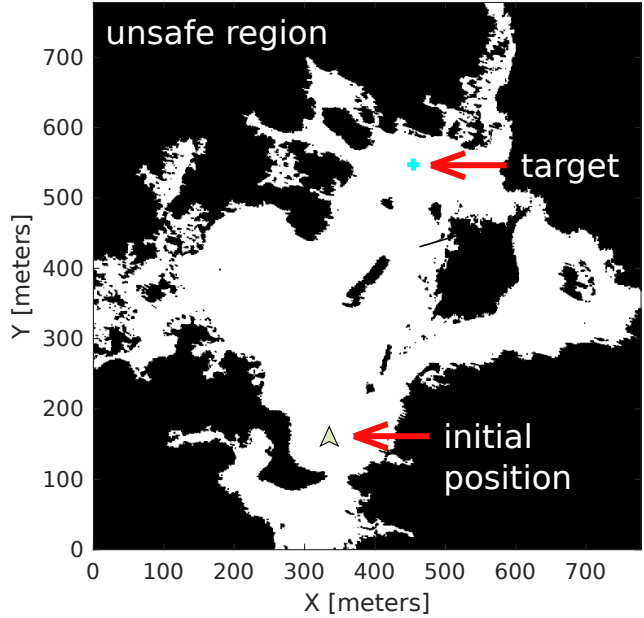}
	\caption{Robot operational space: initial position (yellow arrowhead), target position (cyan), unsafe region (dark area).}
	\label{fig:map}
\end{figure}

The robot's motion, using forward Euler discretization, is:
\begin{align*}
x^+ = 
\begin{bmatrix}
	\I & \I T_s\\
	0 & \I
\end{bmatrix} 
x
+ 
\begin{bmatrix}
0 \\
\I T_s 
\end{bmatrix} 
u
,
\end{align*} 
where the state $x = (p_x,p_y,v_x,v_y) \in\R^4$ is the robot's position and velocity and the input \mbox{$u=(u_1, u_2) \in\R^2$} is the acceleration. The safe set consists of two parts:

1) The time-invariant constraints \mbox{$v_x,v_y \in [-\overline{v},\overline{v}]$} and \mbox{$u_1,u_2 \in [-\overline{u},\overline{u}]$}.

2) The time-varying constraint of $(p_x,p_y)$ within the obstacle-free region, shown by the white nonconvex area in Fig.~\ref{fig:map}. 
The obstacle-free region, denoted by $M(t) \subseteq R^{2}$, is determined by a LiDAR sensor using data up to time $t$. Combining the two constraints, the safe set at time $t$ is:
\begin{align}
S_{xu}(t) = &\{(p_x,p_y,v_x,v_y,u_1,u_2) \mid  (p_x,p_y)\in M(t), \nonumber\\
& v_x,v_y \in [-\overline{v},\overline{v}], ~ u_1,u_2 \in [-\overline{u},\overline{u}]\}.   \nonumber
\end{align}
Since $M(t) \subseteq M(t+1)$, we have $S_{xu}(t) \subseteq S_{xu}(t+1)$, $t\geq 0$.  

The overall control framework is shown in Fig.~\ref{fig:control_framework}. Initially, the map is blank and the path planner generates a reference trajectory assuming no obstacles. At each time $t$, the map is updated based on the latest LiDAR measurements and the path planner checks if the reference trajectory collides with any obstacles in the updated map. If so, it generates a new, collision-free, reference path. Then, the nominal controller provides a candidate input \mbox{$\tilde{u}=(\tilde{u}_1(t),\tilde{u}_2(t))$} tracking the reference path. When updating the reference trajectory, a transient period is needed for the robot to converge to the new reference. Moreover, the path planner cannot guarantee satisfaction of the input constraints. To resolve these issues, we add a supervisory control to the candidate inputs. Based on the updated obstacle-free region $M(t)$, we construct the safe set $S_{xu}(t)$ and compute an implicit CIS $\iCTL(t)$ within $S_{xu}(t)$. To handle the nonconvexity of $S_{xu}(t)$, we first compute a convex composition of $S_{xu}(t)$. When constructing $\iCTL(t)$, we let the reachable set at each time belong to one of the convex components in $S_{xu}(t)$, encoded by mixed-integer linear inequalities. For details see \cite{zexiang2021adhs}. The convex decomposition of $S_{xu}(t)$ becomes more complex over time, which slows down the algorithm. To lighten the computational burden, we replace the full convex composition by the union of the $10$ largest hyper-boxes in $S_{xu}(t)$ as the safe set. Given the constructed implicit CIS $\iCTL(t)$ at time $t$, we supervise the nominal control input $\tilde{u}(t)$ by solving $\mathcal{P}(t,t^*)$ as discussed in Section \ref{subsec:planning}. Note that $P(t,t^*)$ becomes a mixed-integer program as we introduced binary variables for the convex composition of the safe set and, therefore, in the implicit CIS.

\begin{figure}[t!]
	\centering
	\begin{subfigure}{\linewidth}
		\centering
		\includegraphics[width=0.7\textwidth]{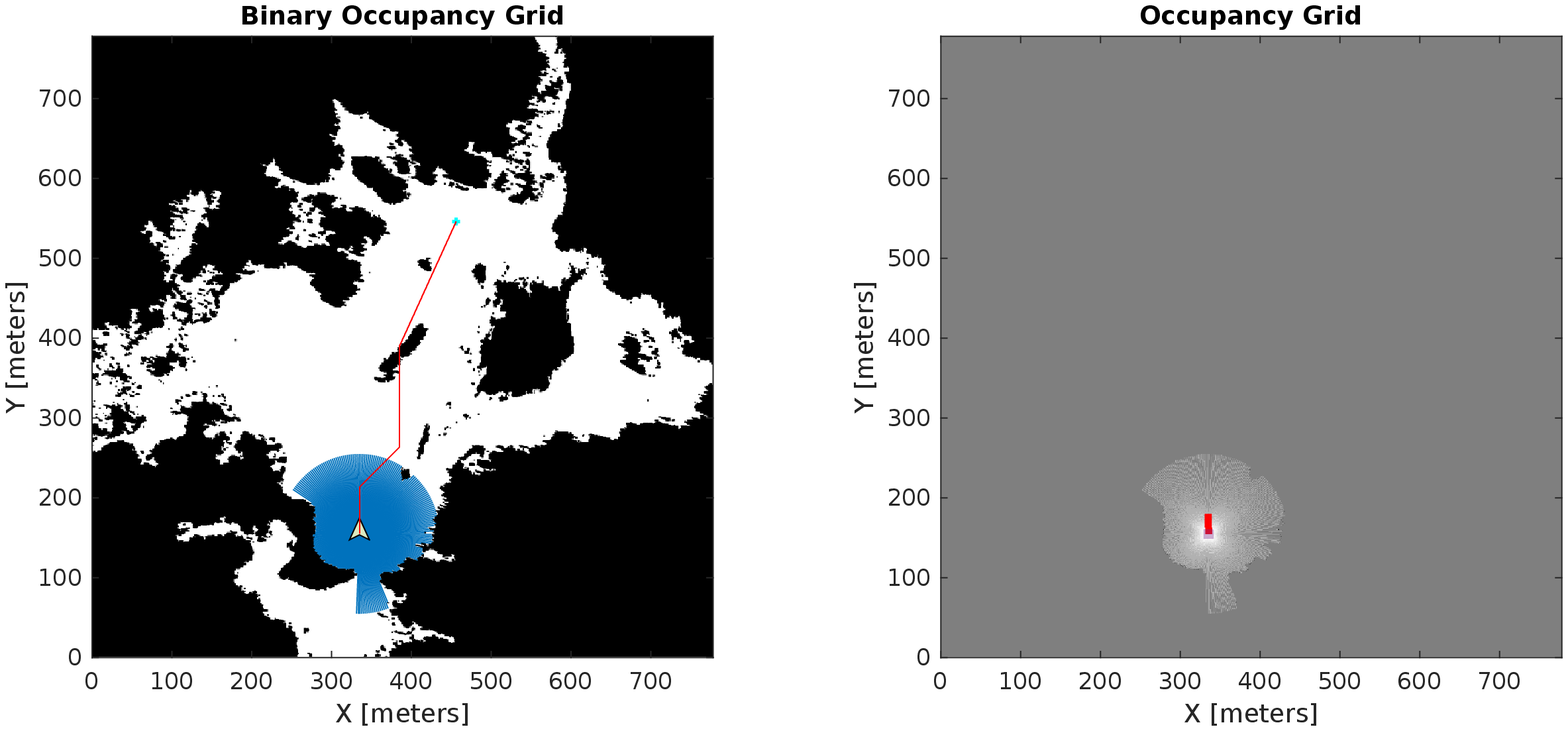}
		\caption{$t=0s$}
		\label{fig:fm0}
	\end{subfigure}
	\begin{subfigure}{\linewidth}
		\centering
		\includegraphics[width=.7\textwidth]{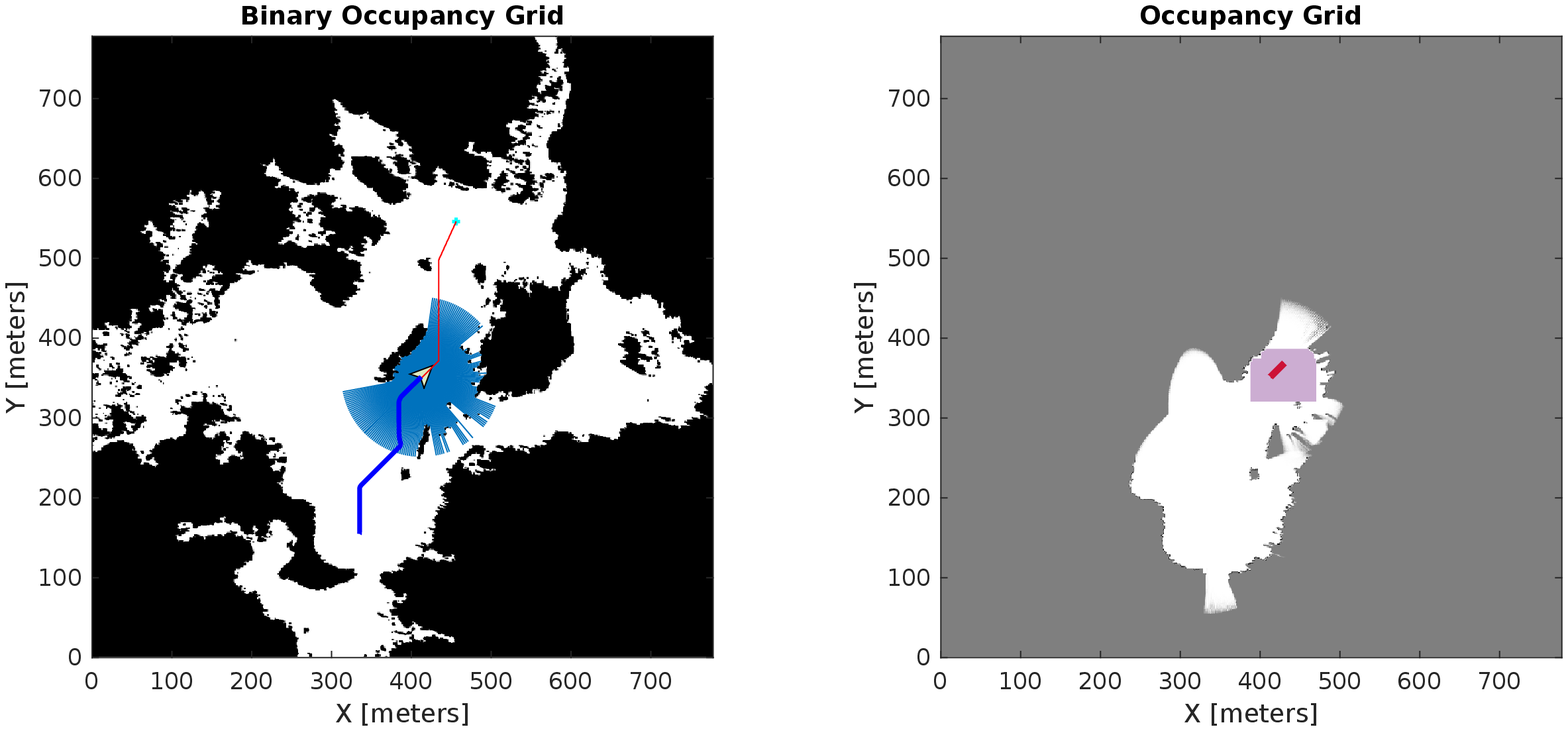}
		\caption{ $t= 40s$}
		\label{fig:fm1}
	\end{subfigure}
	\begin{subfigure}{\linewidth}
		\centering
		\includegraphics[width=.7\textwidth]{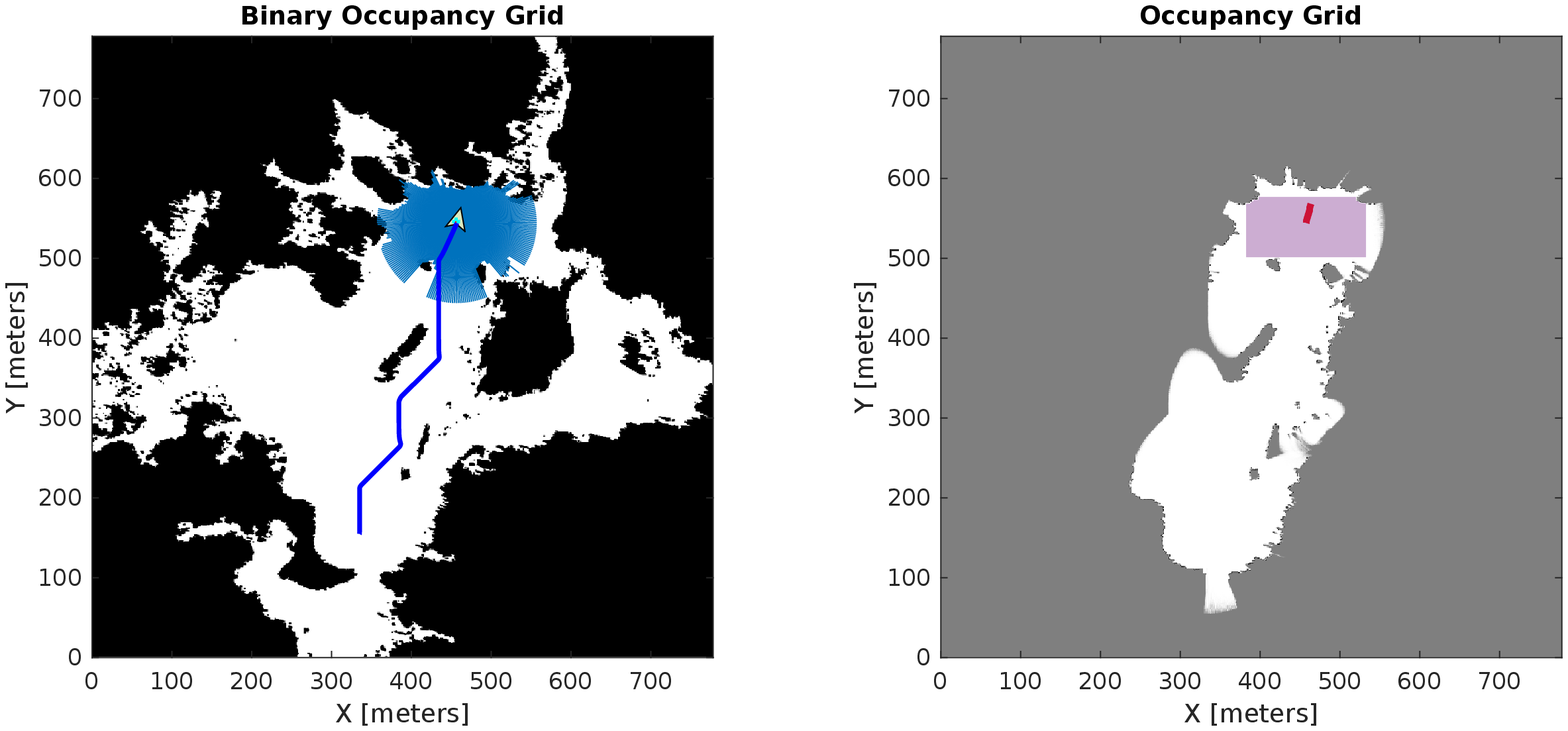}
		\caption{ $t= 78.6 s$}
		\label{fig:fm2}
	\end{subfigure}
	\caption{\small Simulation screenshots at times $t=0s$, $40s$ and $78.6$s. \\
	Left: reference path (red) and actual trajectory (blue); the disk of blue rays is the LiDAR measurements; the arrowhead indicates the position and moving direction of the robot. \\
	Right: obstacle-free region $M(t)$ (white) and unknown region (grey); purple boxes are the $10$ largest boxes in $M(t)$ that contain the current robot position.
	}
	\label{fig:simu}
\end{figure}

In our simulations, we use a linear feedback controller as the nominal controller. The MATLAB Navigation Toolbox is used to simulate a LiDAR sensor with sensing range of $100~m$, update the map, and generate the reference path based on the A* algorithm. The simulation parameters are $\TL=(6,4)$, \mbox{$T_s = 0.1s$}, $\overline{v} = 5 m/s$, $\overline{u} = 5 m/s^2$. The mixed-integer program $\mathcal{P}(t,t^*)$ is implemented via YALMIP \cite{lofberg2004yalmip} and solved by GUROBI \cite{gurobi}. The average computation time for constructing $\iCTL(t)$ and solving $\mathcal{P}(t,t^*)$ at each time step is $2.87 s$. The average computation time shows the efficiency of our method, considering the safe set is nonconvex and being updated at every time step. 

The simulation results are shown in Fig. \ref{fig:simu}. The robot reaches the target region at $t=78.6s$, and thanks to the supervisor, it satisfies the input and velocity constraints, while always staying within the time-varying safe region. As a comparison, when the supervisor is disabled, the velocity constraint is violated at time $t=1.2s$. The full simulation video can be found at \url{https://youtu.be/mB9ir0R9bzM} .

\subsection{Scalability and quality}
\label{subsec:scalability}
In this subsection we illustrate the scalability of the proposed method and compare with other methods in the literature. We consider a system of dimension $n$ as in \eqref{eq:dtls} that is already in Brunovsky normal form \cite{brunovsky1970control}. 
 \begin{align*}
 A_n = 
 \begin{bmatrix}
 	0 & \I 	\\
 	0 & 0
 \end{bmatrix}
 ,
 ~
 B_n = 
 \begin{bmatrix}
 	\0 \\
 	1
 \end{bmatrix}
 ,
 \end{align*}
 where $A_n \in \R^{n \times n}$ and $B_n \in \R^{n}$. This assumption does not affect empirical performance measurements as the transformation that brings a system in the above form is system-dependent and, thus, can be computed offline just once. To generalize the assessment of performance, we generate the safe set as a random polytope of dimension $n$ and we average the results over multiple runs. Moreover, we constraint our input to $[-0.5,0.5]$ and the disturbance to $[-0.1,0.1]$. 

\subsubsection{Scalability of implicit invariant sets}
We begin with the case of no disturbances. Fig.~\ref{fig:bru-hierarchy-no-disturb-2N} and Fig.~\ref{fig:bru-hierarchy-no-disturb-NN} show the times to compute the implicis CIS $\iCmL$ 
for safe sets with $2n$ and $n^2$ constraints respectively. $\iCmL$ can be computed in less than $0.5s$ for systems of size $n=200$ when the safe set has $2n$ constraints, and in around $5s$ for $n=100$ and safe sets with $n^2$ constraints, that is $10000$ constraints in this example. 

\begin{figure}[t!]
\begin{subfigure}{.5\linewidth}
    \includegraphics[width=1.0\linewidth]{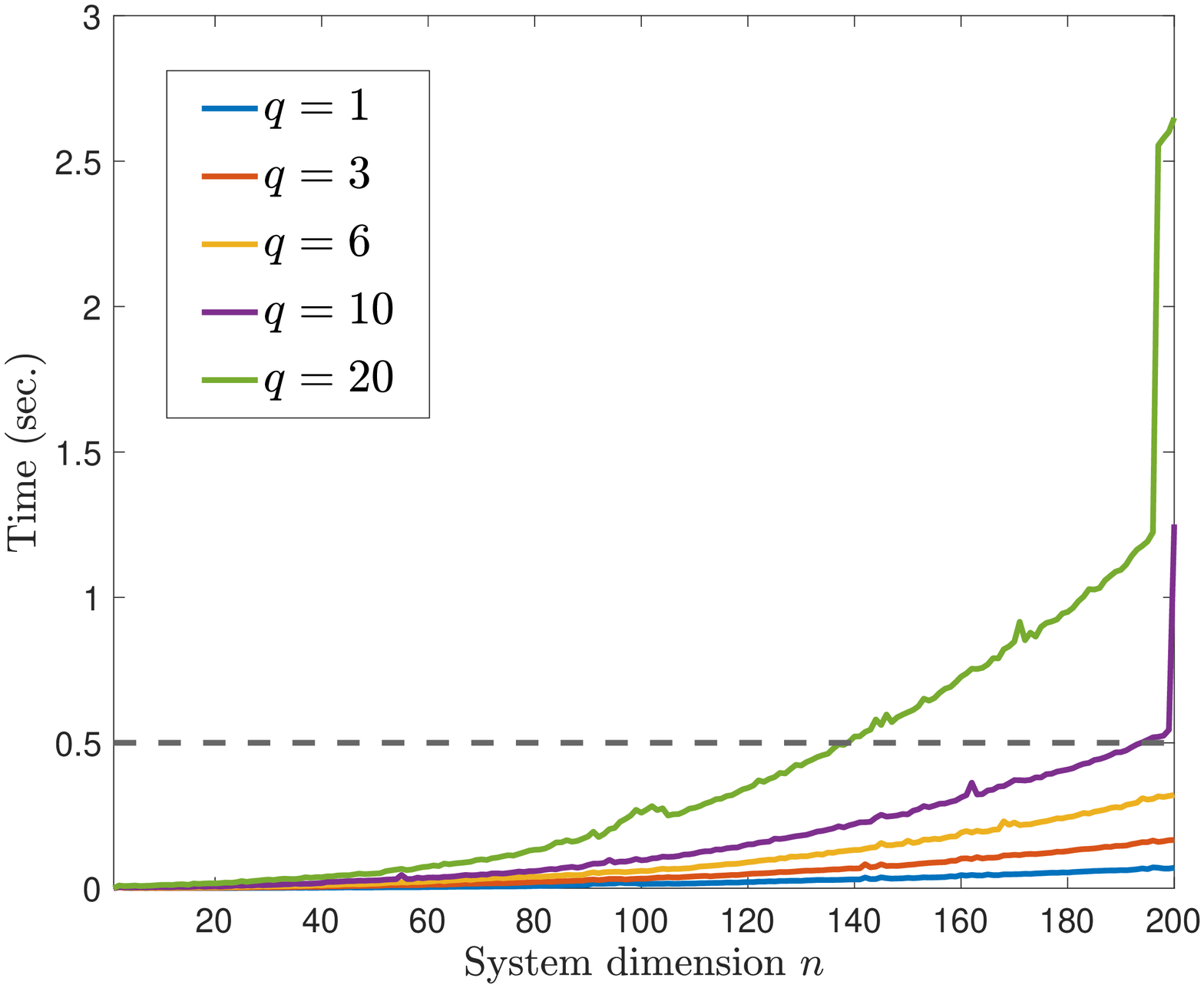}
    \caption{}
    \label{fig:bru-hierarchy-no-disturb-2N}
\end{subfigure}%
\begin{subfigure}{.5\linewidth}
    \includegraphics[width=1.0\linewidth]{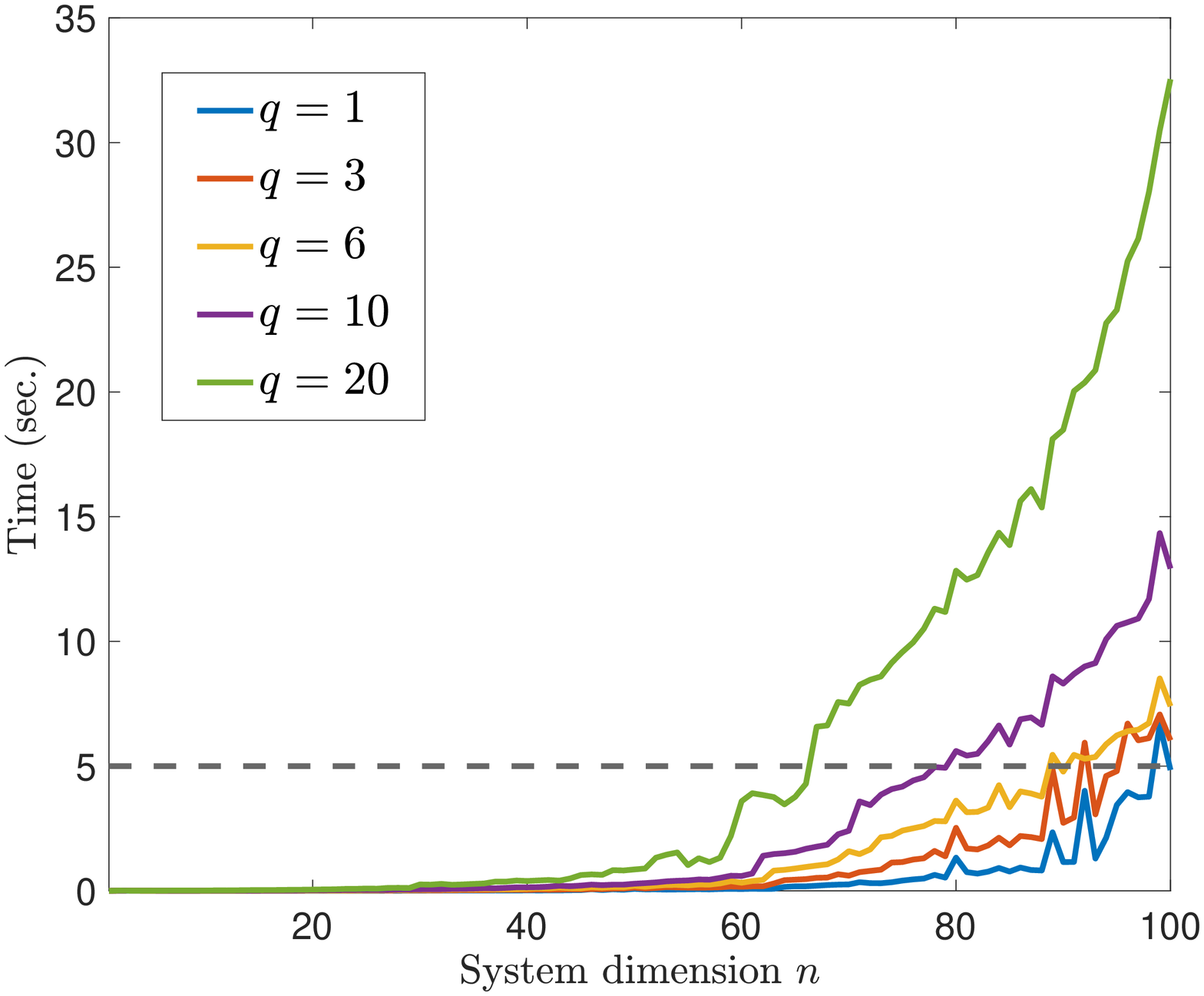}
    \caption{}
    \label{fig:bru-hierarchy-no-disturb-NN}
\end{subfigure}
\caption{Absence of disturbances. Computation times for implicit CISs for different levels $q$ of the full hierarchy, i.e., computing $q$ Implicit CISs per level. (a) Safe sets with $2n$ constraints, $n\leq200$. (b) Safe sets with $n^2$ constraints, $n\leq100$.}
\label{fig:bru-plots-no-disturbance}
\end{figure}

We now proceed to the case where system disturbances are present. In Fig.~\ref{fig:bru-hierarchy-disturb} and Fig.~\ref{fig:bru-single-disturb}, we observe that in the presence of disturbances computations are slower and, actually, are almost identical for different values of $q$. This is attributed to the presence of the Minkowsky difference in the closed-form expression~\eqref{eq:iRCISconditions} that dominates the runtime and depends on the nilpotency index of the system. Still, we are able to compute implicit RCISs in closed-form for systems with up to $20$ states fairly efficiently in this experiment. 

\begin{figure}[t!]
\begin{subfigure}{.5\linewidth}
    \includegraphics[width=1.0\linewidth]{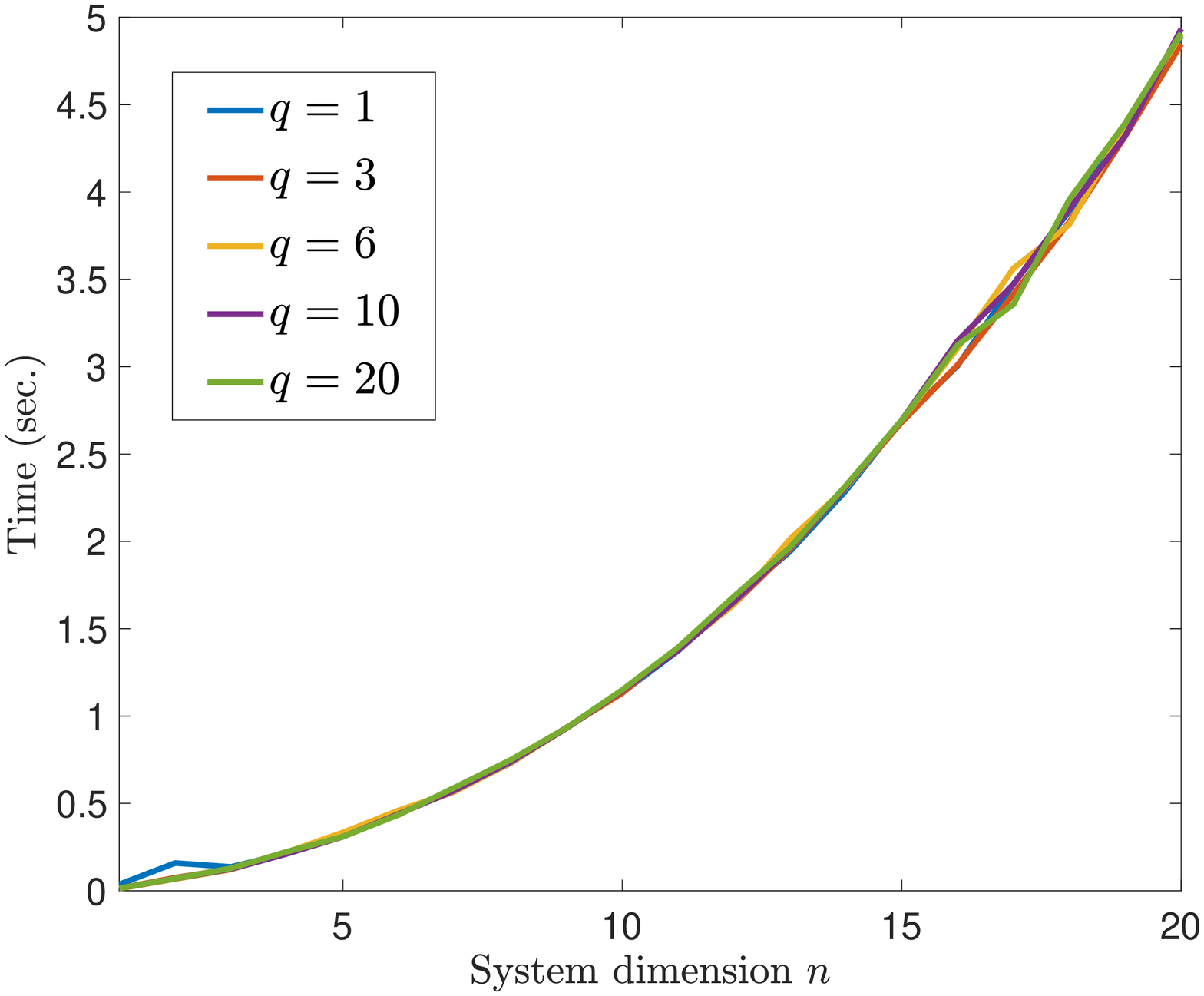}
    \caption{}
    \label{fig:bru-hierarchy-disturb}
\end{subfigure}%
\begin{subfigure}{.5\linewidth}
    \includegraphics[width=1.0\linewidth]{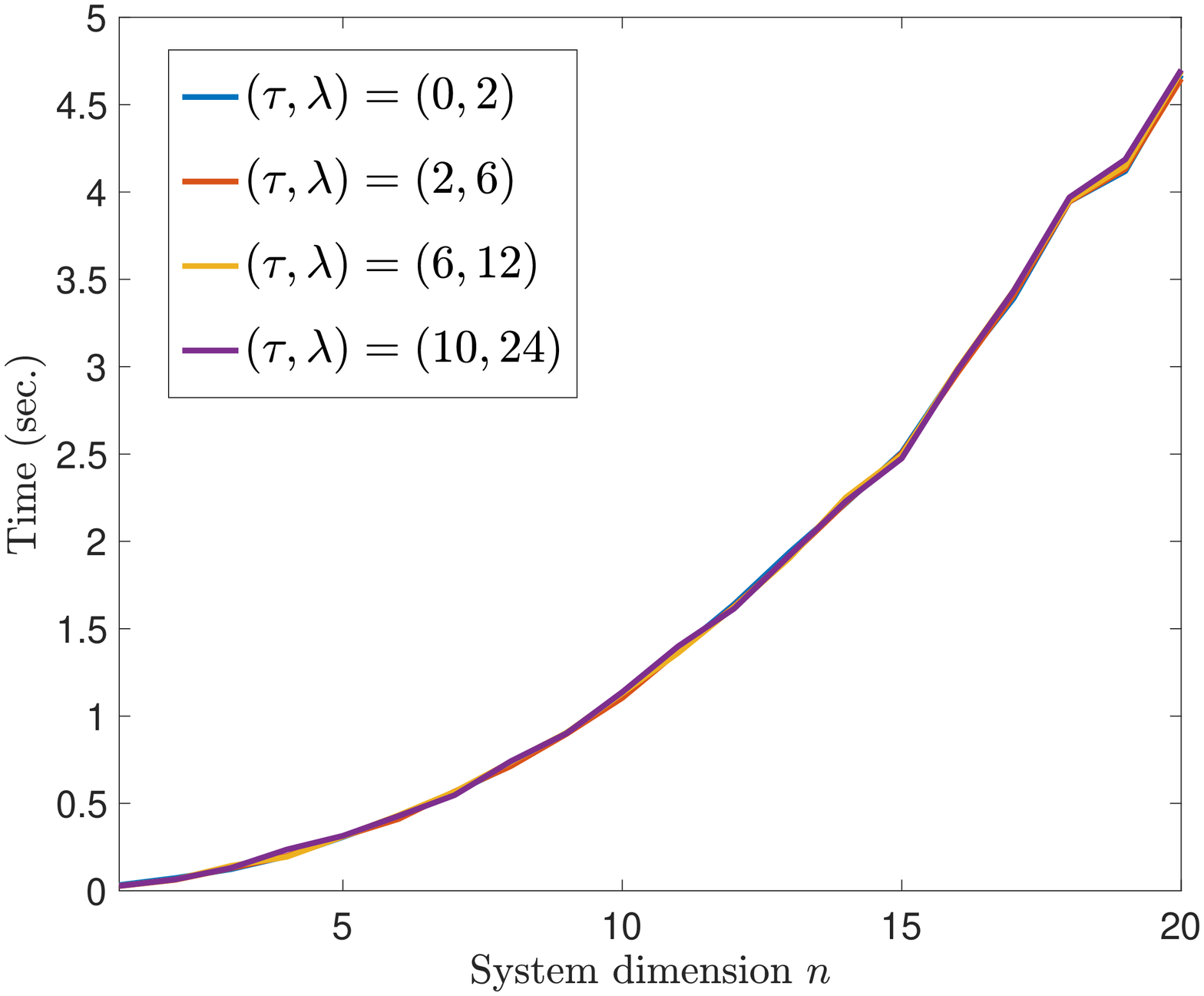}
    \caption{}
    \label{fig:bru-single-disturb}
\end{subfigure}
\caption{Presence of disturbances. Safe sets with $2n$ constraints, $n\leq20$. Computation times for Implicit RCISs. (a) Different levels $q$ of the hierarchy. (d) Individual implicit RCIS, $\iCTL$, for different values of $\TL$.}
\label{fig:bru-plots-disturbance}
\end{figure}

The above results suggest the efficiency and applicability of our approach to scenarios involving online computations, as shown already in Section~\ref{subsec:planning-exp}. Moreover, in our experience, the numerical result of a projection operation, depending on the method used, can be sometimes unreliable. Contrary to this, our closed-form implicit representation does not suffer from such drawback. 

\subsubsection{Quality of the computed sets and comparison to other methods}
We now compare our method with different methods in the literature, both in runtime and quality of the computed sets as measured by the percentage of their volume compared to the Maximal (R)CIS. Even though, we already provided a comprehensive analysis in terms of runtime for our method, we still present a few cases for the shake of comparison. 
We compare our approach to the Multi-Parametric Toolbox (MPT3) \cite{MPT3} that computes the Maximal (R)CIS, $\C_{max}$, the iterative approach in \cite{tahir2015lowcomplnormbound} that computes low-complexity (R)CISs, and the one in \cite{legat2018cishybsys} that computes ellipsoidal CISs. 

\begin{figure}[t!]
	\centering
	\includegraphics[width=0.7125\linewidth]{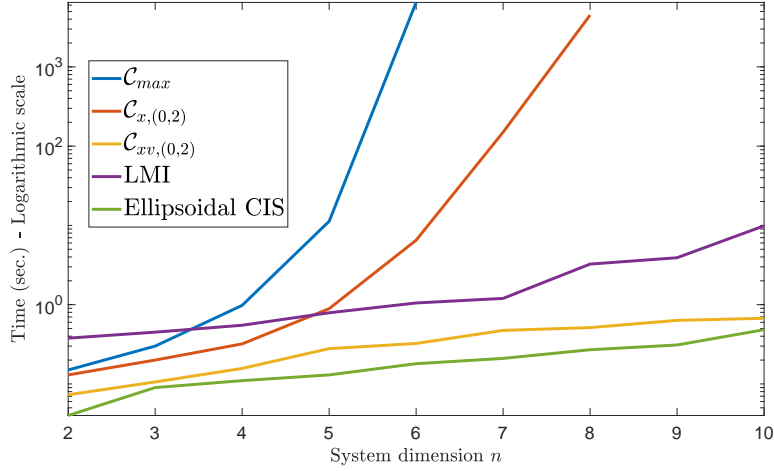}
\caption{Computation times for $\mathcal{C}_{xv,(0,2)}$, its projection $\mathcal{C}_{x,(0,2)}$, the LMI method in \cite{tahir2015lowcomplnormbound}, the ellipsoidal CIS in \cite{legat2018cishybsys}, and $\C_{max}$. Logarithmic scale. Note: \cite{legat2018cishybsys} is evaluated in the absence of disturbances as it considers only nominal systems. For the other methods the performance without disturbance is similar or better.}
\label{fig:comparison-runtimes}
\end{figure}

The runtimes of each method are reported in Fig.~\ref{fig:comparison-runtimes}. The difficulty of computing $\C_{max}$ is apparent from the steep corresponding curve. 
The low-complexity methods in \cite{tahir2015lowcomplnormbound} and \cite{legat2018cishybsys} are considerably faster, and \cite{legat2018cishybsys} is slightly faster than even our implicit representation. However, our sets are superior in quality as we detail next. 

First, in the absence of disturbances, the relative volume of the computed sets with respect to $\C_{max}$ is presented in Table~\ref{table:bru-vol-no-disturbance}. 
Since for $n\geq7$ MPT3 does not terminate after several hours and the computed set before termination is not invariant, we present the relative volumes only for \mbox{$2\leq n \leq 6$}. Our method returns a very close approximation of $\C_{max}$ even with small values of $\TL$ and computes substantially larger sets compared to the other techniques. This supports our theoretical result in Corollary~\ref{cor:nominal_converge}. In other words, our implicit representation retains the best out of two worlds: computational efficiency and close approximations of $\C_{max}$. 

\begin{table}[t!]
\caption{
Absence of disturbances. Volume percentage with respect to the Maximal CIS. \\ 
Algorithms: Our method for different implicit CISs $\iCTL$, the LMI method in \cite{tahir2015lowcomplnormbound}, and the method in \cite{legat2018cishybsys} computing ellipsoidal CISs. (S) denotes a singleton set. 
}
\label{table:bru-vol-no-disturbance}
\centering
\begin{tabular}{c|ccccccc}
\hline
				&      & \multicolumn{2}{c}{Our method}         & \makecell{LMI method\\ \cite{tahir2015lowcomplnormbound}} & \makecell{Ellipsoidal\\CIS method \cite{legat2018cishybsys}} \\
\hline
\makecell{System\\dimension}    & & $\C_{xv,(0,2)}$ & $\C_{xv,(4,2)}$ &  & \\
\hline
$n=2$                       &    & 100 & 100 & 42.43 &  45.69      \\
\hline
$n=3$                       &    & 100 & 100 & 16.31 & 24.66        \\
\hline
$n=4$                       &    & 99.92 & 100   & 3.69 & 14.41   \\
\hline
$n=5$                       &    & 99.75 & 100   &  0.47  & 10.50  \\
\hline
$n=6$                       &    & 97.81 & 100   & 0 (S) & 3.89  \\
\hline
\end{tabular}
\end{table}

\begin{table}[t!]
\caption{
Presence of disturbances. Volume percentage with respect to the Maximal RCIS. \\
Algorithms: Our method for different implicit RCISs $\iCTL$ and the LMI method in \cite{tahir2015lowcomplnormbound}.  (S) denotes a singleton set. 
}
\label{table:bru-vol-disturbance}
\centering
\def\arraystretch{}
\begin{tabular}{c|cccccccc}
\hline
Volume (\%)              &      & \multicolumn{3}{c}{Our method}         & \makecell{LMI method \cite{tahir2015lowcomplnormbound}}  \\
\hline
\makecell{System\\dimension}    & & $\C_{xv,(0,2)}$ & $\C_{xv,(2,2)}$ & $\C_{xv,(4,2)}$ &  \\
\hline
$n=2$                       &    & 100 & 100 & 100 & 31.99       \\
\hline
$n=3$                       &    & 98.24 & 99.67 & 99.96 & 16.35        \\
\hline
$n=4$                       &    & 99.02 & 99.42 & 99.88   & 4.36	\\
\hline
$n=5$                       &    & 98.75 & 99.74 & 99.81   &  3.64  \\
\hline
$n=6$                       &    & 91.17 & 96.07 & 97.91   & 0 (S)  \\
\hline
\end{tabular}
\end{table}

In the presence of disturbances, the results are similar and are reported in Table~\ref{table:bru-vol-disturbance}, where we omit \cite{legat2018cishybsys} that is not designed for the case of disturbances. Theorem~\ref{thm:convergence} proves that our method converges to its outer bound $\C_{outer,\nu}$. Here, we can appreciate that empirically $\C_{outer,\nu}$ approximates very closely $\C_{max}$, even in the presence of disturbances, based on the size of the sets computed by our method. However, the gap between $\C_{outer,\nu}$ and $\C_{max}$ depends on the size of the disturbance as shown next. 

We illustrate how the size of the disturbance set affects our performance. We fix the safe set to be a random polytope in $\R^4$ and constraint the input to $[-0.5,0.5]$. The disturbance set is \mbox{$W=[-\overline{w},\overline{w}]$} and we increase $\overline{w}$ as in Table~\ref{table:increase-disturbance}. Recall the nominal system $\overline{\Sigma}$ and the nominal safe set \mbox{$\overline{S}_{xu} = S_{xu}-\overline{W}_{\infty}$}, and let $\overline{\Delta}_{xu}$ be the set of fixed points of $\overline{\Sigma}$, which is in Brunovsky normal form. We can show that \mbox{$\overline{\Delta}_{xu} = \{(x,u)\in\R^4\times\R | x_1 = x_2 = x_3 = x_4 = u\}$}.  As Table~\ref{table:increase-disturbance} details, by increasing the size of $W$ the our RCIS shrinks at a faster rate compared to $C_{max}$, until finally $\overline{S}_{xu}$ is empty and, hence, does not contain any fixed points from $\overline{\Delta}_{xu}$. This is when the set we compute becomes empty as well.

\begin{table}[t!]
\caption{
Increasing the size of the disturbance set \mbox{$W=[-\overline{w},\overline{w}]$}. 
Volume percentage of $\C_{x,(2,2)}$ 
with respect to the Maximal RCIS and volume percentage of $\overline{S}_{xu}$ with respect to $S_{xu}$. (NE) set is nonempty. (E) set is empty. 
}
\label{table:increase-disturbance}
\centering
\def\arraystretch{}
\begin{tabular}{c|cccccccc}
\hline
$\overline{w}$    & 0.05 & 0.10 & 0.15 & 0.20 & 0.25 & 0.30 & 0.35 & 0.40 \\
\hline
\makecell{$\frac{vol~\C_{x,(2,2)}}{vol~\C_{max}}$}  & 99.9 & 99.7 & 99.3 & 98.1 & 91.8 & 10.5 & \makecell{$\Cx$\\empty} & \makecell{$\C_{max}$\\empty}  \\
\hline
\makecell{$\frac{vol~\overline{S}_{xu}}{vol~S_{xu}}$}   & 63.9 & 38.2 & 20.5 & 9.2 & 2.6 & 0.1 & \makecell{$\overline{S}_{xu}$\\empty} & \makecell{$\overline{S}_{xu}$\\empty} \\
\hline
\makecell{$\overline{S}_{xu} \cap \overline{\Delta}_{xu}$}   & NE & NE & NE & NE & NE & NE & E & E\\
\hline
\end{tabular}
\end{table}

\section{Related literature}
\label{subsec:litreview}
Recent works by the authors \cite{anevlavis2019cis2m,anevlavis2020simple,anevlavis2021enhanced} develop methods constructing implicit RCISs in closed-form. These approaches consider different collections of periodic input sequences, which can be viewed as special instances of the parameterization proposed here. Furthermore, in this work we provide theoretical performance results, both for completeness and convergence, which extend to the previous methods as special cases.
The concept of implicit RCISs is also explored in \cite{fiacchini2017ciseasy} for nominal systems and in \cite{wintenberg2020implicit} for systems with disturbances. However, different from our method, \cite{fiacchini2017ciseasy, wintenberg2020implicit} both need to check a sufficient condition on set recurrence by LPs and do not provide completeness guarantees. 

In addition to the aforementioned methods, a plethora of works have attempted to alleviate the poor scalability and the absence of termination guarantees of the standard method for computing the Maximal CIS of discrete-time systems introduced in \cite{bertsekas1972infreach}. The following list is not exhaustive. 

One line of work \cite{korda2014cispolysys, oustry2019innerapproxcis} focuses on outer and inner approximations of the Maximal CIS by solving either LPs or QPs. The resulting sets, however, are not always invariant. Comparing to these works, our method provides sets that are always guaranteed to be invariant and, given our closed-form expression, scales better with the system dimension.

Other methods compute exact ellipsoidal CISs efficiently and, thus, offer improved scalability, such as \cite{legat2018cishybsys} which solves Semi-Definite Programs (SDP) for a class of hybrid systems. Nevertheless, the resulting ellipsoidal sets are generally small. This is backed by our comparison studies, which show that even though \cite{legat2018cishybsys} computes very efficiently exact CISs, our implicit CIS offers similar computational performance, but substantially better quality in terms of approximating the Maximal CIS. In addition, for online control problems, like MPC and supervisory control, polytopes are preferred to ellipsoids, as they result in LPs or QPs, which are solved more efficiently compared to Quadratically Constrained Quadratic Programs (QCQP) that stem from ellipsoids. 

In the presence of bounded disturbances, when the set of safe states are polytopes, \cite{rungger2017rcislinsys} computes inner and outer approximations of the Maximal RCIS for linear systems. However, this iterative method suffers from the usual problem of performing an expensive projection operation in between iterations, which hinders its applicability in practice.

Ideas similar to ours, in the sense that finite input sequences are used, have been explored in the context of MPC \cite{mayne2005robustmpc}. The goal there is to establish asymptotic stability of a linear system, whereas in our case we exploit finite input sequences that describe the proposed control behavior, which leads to a closed-form expression for an implicit representation of controlled invariant sets. Other popular approaches first close the loop with a linear state-feedback control law, and then compute an invariant set of the closed-loop system. Under this umbrella, an idea close to ours is found in \cite{lazar2015finstepmpc}, where recurrent sets are computed in the context of MPC without disturbances. This can be understood as a special case of our eventually periodic approach. 

In a similar spirit, i.e., by restricting to linear state-feedback control laws, the following works focus on reducing the computational cost and employ iterative procedures to compute low- or fixed-complexity RCISs and their associated feedback gains. In \cite{tahir2015lowcomplnormbound} low-complexity RCISs are found via SDPs under norm-bounded uncertainties. More recently, \cite{gupta2019lowcompuncertainparams, gupta2019fullcompluncertainparams} compute low- and fixed-complexity RCISs respectively for systems with rational parameter dependence. The complexity, i.e., the number of inequalities of the set, in \cite{gupta2019lowcompuncertainparams} is twice the number of states, while \cite{gupta2019fullcompluncertainparams} is more flexible as the complexity can be pre-decided. These methods assume the RCIS to be symmetric around the origin, whereas we make no assumptions on the RCIS. Arguably, pre-deciding the complexity is valuable for applications, such as MPC, but it can be very conservative. Increasing the set complexity to obtain larger sets hinders performance of said iterative methods. In comparison, we offer an alternative way to obtain larger sets, by increasing the transient and/or the period of the eventually periodic input parameterization. This bears minimal computation impact due to the derived closed-form expression. 

The work of \cite{munir2018highdegreelyap} computes larger controlled contractive sets of specified degree for nominal linear systems by solving Sum Of Squares (SOS) problems, but requires prior knowledge of a contractive set. Their scalability is also limited by the size of the SOS problems and so is its extension to handle polytopic uncertainty, which significantly increases the SOS problem size. Again, our method offers improved scalability along with the ability to increase the size of the computed set with minimal performance impact, as is backed by our experiments.

\section{Appendix}
\subsection{Claims of Theorem~\ref{thm:convergence}}
\label{sec:proof_of_claims}

\emph{Proof of Claim 1}: Since $\overline{S}_{xu}$ contains the origin, we have $\overline{W}_{ \infty}\times \{0\}  \subseteq S_{xu}$.  Since $0\in W$, it is easy to verify from \eqref{eq:acc_dist_set} and \eqref{eq:min_rpis} that $\overline{W}_{k} \subseteq \overline{W}_{ \infty}$ for all $k\geq 1$. Thus, $\overline{W}_{k}\times \{0\} \subseteq S_{xu} $ for all $k\geq 1$. According to \eqref{eq:reachableset}, if $x=0$ and $u_{t}=0$ for all $t\geq 0$, the reachable set $\left( \reach{x}{\{u_{i}\}_{i=0}^{t-1}},u_{t}\right)= \overline{W}_{t}\times \{0\} \subseteq S_{xu}$. Thus, $0\in \C_{xv,0}$. 

\emph{Proof of Claim 2}: Recall from \eqref{eq:C_nominal_pre} that $\C_{outer,\nu}$ is:
    \begin{align*} 
    \begin{aligned}
	\C_{outer,\nu} = \Big\{x \in \R^{n} \mid \exists \{u_{i}\}_{i=0}^{\nu-1} \in \R^{m\nu}, &\left(\reach{x}{\{u_{i}\}_{i=0}^{t-1}},u_{t}\right) \subseteq S_{xu}, t = 0,\dots, \nu-1, \\
	&\reach{x}{\{u_{i}\}_{i=0}^{\nu-1}} \subseteq \overline{\C}_{\max} + \overline{W}_{\infty} 
	\Big\}.
    \end{aligned}
    \end{align*}
Due to Lemma \ref{thm:inv_sum}, $\reach{x}{\{u_{i}\}_{i=0}^{\nu-1}} \subseteq \overline{\C}_{\max} + \overline{W}_{\infty}$ if and only if $\sum_{i=0}^{\nu-1} A^{\nu-1-i}Bu_{i}\in \overline{\C}_{max}$. Based on this observation, it is easy to verify that $\C_{outer, \nu} = \pi_{n} (\C_{xv,max})$ where $\C_{xv,max} = \C_{xv,0}\cap ( \R^{n} \times \mathcal{U}(\overline{\C}_{max}))$.

\emph{Proof of Claim 3}: We show that $\widehat{\C}_{xv, (\tau, \lambda)}=  \pi_{n+\nu m}(\C_{xv,(\tau, \lambda)})$. Using the matrices $H$ and $P$ as in \eqref{eq:lassobehavior}, by the definition of $\iCTL$ and Lemma \ref{thm:inv_sum}, we can write $\C_{xv,(\tau, \lambda)}$ as: 
\begin{align} 
\label{eq:53} 
\begin{split}
	\C_{xv,(\tau, \lambda)} = \Big\{(x_0,u_{0:\tau+\lambda-1}) \mid &\left(\reach{x}{\{u_{i}\}_{i=0}^{t-1}},u_{t}\right) \subseteq S_{xu}, t= 0, \cdots, \nu-1, \\
					&(\mathcal{R}_{\overline{\Sigma}}(\sum_{i=1}^{v}A^{i-1}Bu_{v-i}, \{u_{v+i}\}_{i=0}^{k-1}),u_{v+k})\in \overline{S}_{xu}, k=0, \cdots, \tau + \lambda-1\Big\}. 
\end{split}
\end{align}
By \eqref{eq:53} , the projection $\pi_{n+\nu m}(\C_{xv,(\tau, \lambda)})$ is: 
\begin{align} \label{eq:54} 
\begin{split}
	\pi_{n+\nu m}(\C_{xv,(\tau, \lambda)}) = \Big\{(x_0,u_{0:\nu-1}) \mid &\exists u_{\nu:\tau+ \lambda-1}, \left(\reach{x}{\{u_{i}\}_{i=0}^{t-1}},u_{t}\right) \subseteq S_{xu}, t= 0, \cdots, \nu-1,\\
							 &(\mathcal{R}_{\overline{\Sigma}}(\sum_{i=1}^{\nu}A^{i-1}Bu_{\nu-i}, \{u_{\nu+i}\}_{i=0}^{k-1}),u_{\nu+k})\in \overline{S}_{xu}, k=0, \cdots, \tau + \lambda-1\Big\}.
\end{split}
\end{align}
Again, using the matrices $H$ and $P$ as in \eqref{eq:lassobehavior}, by the definition of $\overline{\C}_{x,(\tau-\nu, \lambda)}$, we have: 
\begin{align} 
\label{eq:55} 
	\overline{\C}_{x,(\tau-\nu, \lambda)} = &\Big\{x_0\in \R^{n} \mid  \exists u_{0:\tau-\nu + \lambda}, ( \mathcal{R}_{\overline{\Sigma}}(x_0, \{u_{i}\}_{i=0}^{t-1} ),u_{t})\in \overline{S}_{xu}, t= 0, \cdots, \tau + \lambda-1 \Big\}. 
\end{align}
Comparing the right hand sides of \eqref{eq:54} and \eqref{eq:55}, we have: 
\begin{align} \label{eq:56} 
\begin{split}
	\pi_{n+\nu m}(\C_{xv,(\tau, \lambda)}) =\Big\{(x_0,u_{0:\nu-1}) \mid &\left(\reach{x}{\{u_{i}\}_{i=0}^{t-1}},u_{t}\right) \subseteq S_{xu}, t= 0, \cdots, \nu-1, \\
	&\sum_{i=1}^{\nu}A^{i-1}Bu_{\nu-i}\in \overline{\C}_{x,(\tau-\nu, \lambda)}\Big\}. 
\end{split}
\end{align}
Note that $\C_{xv,0}$ and $\mathcal{U}(\overline{\C}_{x,(\tau-\nu, \lambda)})$ respectively impose the first and second constraints on $(x_0, u_{0:\nu-1})$ on the right hand side of \eqref{eq:56}. Thus, $\pi_{n+\nu m}(\C_{xv, (\tau, \lambda)})$ is equal to the intersection of $\C_{xv,0}$ and $\mathcal{U}(\overline{\C})_{x, (\tau-\nu, \lambda)}$. That is: 
\begin{align} \label{eq:57} 
    \widehat{\C}_{xv,(\tau, \lambda)} = \pi_{n+\nu m}(\C_{xv,(\tau, \lambda)}).
\end{align}
Since \eqref{eq:57} implies \eqref{eq:45}, the third claim is proven. 

\emph{Proof of Claim 4}: We define the $k$-step null-controllable set $\C_{k}$ as the set of states of $\overline{\Sigma}$ that reach the origin at $k$th step under the state-input constraints $S_{xu}$: 
\begin{align} \label{eq:58} 
	\C_{k} = \Big\{x \in \R^{n}  &\mid  \exists u_{0:k-1}\in \R^{k m},  (\mathcal{R}_{\overline{\Sigma}}(x, \{u_{i}\}_{i=0}^{t-1}), u_{t} )\in \overline{S}_{xu}, t=0, \cdots, k-1, \mathcal{R}_{\overline{\Sigma}}(x, \{u_{i}\}_{i=0}^{k-1}) =0\Big\}. 
\end{align}
Obviously, $\C_{0} = \{0\}$. Since $A^{\nu} = 0$ and the fixed point $(0,0) \in \R^{n}\times \R^{m}$ is in the interior of  $\overline{S}_{xu}$, there exists an $ \epsilon >0$ such that the $ \epsilon$-ball $B_{ \epsilon}(0)$ at the origin satisfies that for $u_{0:\nu-1} = 0\in \R^{\nu m}$ and for all $t\in [0,  k-1]$: 
\begin{align}
	\begin{split} \label{eq:59} 
    & (\mathcal{R}_{\overline{\Sigma}}(B_{ \epsilon}(0), \{u_{i}\}_{i=0}^{t-1}), u_{t} ) = (A^{t}B_{ \epsilon}(0),0) \subseteq \overline{S}_{xu},\\ 
    &\mathcal{R}_{\overline{\Sigma}}(x, \{u_{i}\}_{i=0}^{\nu-1} =A^{\nu} B_{ \epsilon}(0) = 0.
	\end{split}
\end{align}
By \eqref{eq:59} and the definition of $\C_{k}$, $B_{ \epsilon}(0)$ is contained by $\C_{\nu}$, and thus $\C_0 = \{0\} $ is contained in the interior of $\C_{\nu}$. Then, by Theorem 1 in \cite{liu22conv}, since $\C_0$ is contained in the interior of $\C_{\nu}$, there exists $\tau_2 \geq 0$, $c_2\in [0,1]$ and $a \in [0,1)$ such that for all $k\geq \tau_2$, the Hausdorff distance $d(\C_{k}, \overline{\C}_{max})$ satisfies that: 
\begin{align} 
\label{eq:60} 
    d(\C_{k}, \overline{\C}_{max}) \leq c_2 a^{k}.
\end{align}
Furthermore, let $k=\tau$. For any $x\in \C_{\tau}$ and the corresponding $u_{0:\tau-1}$ satisfying the constraints on the right hand side of \eqref{eq:58},  it is easy to check that $(x,u_{0:\tau-1},0)\in \R^{n}\times \R^{(\tau+\lambda)m}$ is contained in $\overline{\C}_{xv, (\tau, \lambda)}$. Thus, we have for all $\tau \geq 0$: 
\begin{align} \label{eq:61} 
    \C_{\tau} \subseteq \overline{\C}_{x,(\tau, \lambda)} \subseteq \overline{\C}_{max}.
\end{align}
Thus, by \eqref{eq:60} and \eqref{eq:61}, for any $\tau \geq \tau_2$, the Hausdorff distance $d(\overline{\C}_{x,(\tau, \lambda)}, \overline{\C}_{max} )$ satisfies:  
\begin{align} \label{eq:62} 
    d(\overline{\C}_{x,(\tau, \lambda)}, \overline{\C}_{max} ) \leq c_2a^{\tau}.
\end{align}
From the properties of Hausdorff distance, \eqref{eq:62} implies that: 
\begin{align} \label{eq:63} 
    \overline{\C}_{max} \subseteq \overline{\C}_{x,(\tau, \lambda)} + B_{ c_2 a^{\tau}}(0),
\end{align}
where $B_{c_2a^{\tau}}(0)$ is the ball at origin with radius $c_2 a^{\tau}$. Recall that $\C_{\nu}$ contains a $ \epsilon$-ball $B_{ \epsilon}(0)$ for some $ \epsilon >0$. Since $\C_{\nu} \subseteq \overline{\C}_{max}$, we have $(c_2 a^{\tau}/ \epsilon)\overline{\C}_{max} \supseteq B_{ c_2a^{\tau}}(0)$. Thus, by \eqref{eq:63}, we have for any $\tau \geq \tau_2$: 
\begin{align} 
\label{eq:64} 
    \overline{\C}_{max} \subseteq \overline{\C}_{x,(\tau, \lambda)} + \frac{c_2 a^{\tau}}{ \epsilon} \overline{\C}_{max}.
\end{align}
Select a big enough $\tau_1$ such that $\tau_1 \geq \tau_2$ and $c_2 a^{\tau_1} \leq  \epsilon$. Then, by Lemma \ref{thm:inv_sum} and \eqref{eq:64}, we have for any $\tau \geq  \tau_1$: 
\begin{align*}
\overline{\C}_{x,(\tau, \lambda)} \supseteq	(1-c_0 a^{\tau})\overline{\C}_{max}, 
\end{align*}
where $c_0 = \frac{c_2}{\epsilon}$. Thus, the fourth claim is proven.


\bibliographystyle{alpha}
\bibliography{jMasterBib}

\newcommand{\etalchar}[1]{$^{#1}$}
\begin{thebibliography}{RKKM05}

\bibitem[ALOT21]{anevlavis2021enhanced}
Tzanis Anevlavis, Zexiang Liu, Necmiye Ozay, and Paulo Tabuada.
\newblock An enhanced hierarchy for (robust) controlled invariance.
\newblock In {\em 2021 American Control Conference (ACC)}, pages 4860--4865,
  2021.

\bibitem[AM97]{antsaklis1997linear}
Panos~J Antsaklis and Anthony~N Michel.
\newblock {\em Linear systems}, volume~8.
\newblock Springer, 1997.

\bibitem[AT19]{anevlavis2019cis2m}
T.~{Anevlavis} and P.~{Tabuada}.
\newblock Computing controlled invariant sets in two moves.
\newblock In {\em 2019 IEEE 58th Conference on Decision and Control (CDC)},
  pages 6248--6254, 2019.

\bibitem[AT20]{anevlavis2020simple}
Tzanis Anevlavis and Paulo Tabuada.
\newblock A simple hierarchy for computing controlled invariant sets.
\newblock In {\em Proceedings of the 23rd International Conference on Hybrid
  Systems: Computation and Control}, HSCC '20, New York, NY, USA, 2020.
  Association for Computing Machinery.

\bibitem[BBB{\etalchar{+}}19]{bajcsy2019efficient}
Andrea Bajcsy, Somil Bansal, Eli Bronstein, Varun Tolani, and Claire~J Tomlin.
\newblock An efficient reachability-based framework for provably safe
  autonomous navigation in unknown environments.
\newblock In {\em 2019 IEEE 58th Conference on Decision and Control (CDC)},
  pages 1758--1765. IEEE, 2019.

\bibitem[Ber72]{bertsekas1972infreach}
Dimitri Bertsekas.
\newblock Infinite time reachability of state-space regions by using feedback
  control.
\newblock {\em Automatic Control, IEEE Transactions on}, AC-17:604 -- 613, 11
  1972.

\bibitem[Bru70]{brunovsky1970control}
Pavol Brunovsk{\'y}.
\newblock A classification of linear controllable systems.
\newblock {\em Kybernetika}, 6:173--188, 1970.

\bibitem[BV04]{boyd2004convex}
Stephen Boyd and Lieven Vandenberghe.
\newblock {\em Convex Optimization}.
\newblock Cambridge University Press, 2004.

\bibitem[CDS02]{caravani2002doubly}
Paolo Caravani and Elena De~Santis.
\newblock Doubly invariant equilibria of linear discrete-time games.
\newblock {\em Automatica}, 38(9):1531--1538, 2002.

\bibitem[{CON}]{arena}
{CONIX Research Center}.
\newblock {Augmented Reality Edge Networking Architecture – ARENA}.

\bibitem[DGCS71]{glover1971linsysdist}
J~Duncan~Glover and Fred C.~Schweppe.
\newblock Control of linear dynamic systems with set constrained disturbances.
\newblock {\em Automatic Control, IEEE Transactions on}, 16:411 -- 423, 11
  1971.

\bibitem[FA17]{fiacchini2017ciseasy}
Mirko Fiacchini and Mazen Alamir.
\newblock Computing control invariant sets is easy.
\newblock {\em CoRR}, abs/1708.04797, 2017.

\bibitem[GF19]{gupta2019fullcompluncertainparams}
A.~{Gupta} and P.~{Falcone}.
\newblock Full-complexity characterization of control-invariant domains for
  systems with uncertain parameter dependence.
\newblock {\em IEEE Control Systems Letters}, 3(1):19--24, 2019.

\bibitem[GKF19]{gupta2019lowcompuncertainparams}
Ankit Gupta, Hakan K{\"o}ro{\u{g}}lu, and Paolo Falcone.
\newblock Computation of low-complexity control-invariant sets for systems with
  uncertain parameter dependence.
\newblock {\em Automatica}, 101:330--337, 2019.

\bibitem[GO20]{gurobi}
LLC Gurobi~Optimization.
\newblock Gurobi optimizer reference manual, 2020.

\bibitem[GU19]{grzybowski2019order}
Jerzy Grzybowski and R~Urba{\'n}ski.
\newblock Order cancellation law in the family of bounded convex sets.
\newblock {\em Journal of Global Optimization}, pages 1--12, 2019.

\bibitem[HKJM13]{MPT3}
M.~Herceg, M.~Kvasnica, C.N. Jones, and M.~Morari.
\newblock {Multi-Parametric Toolbox 3.0}.
\newblock In {\em Proc.~of the European Control Conference}, pages 502--510,
  Z\"urich, Switzerland, July 17--19 2013.
\newblock \url{http://control.ee.ethz.ch/~mpt}.

\bibitem[KHJ14]{korda2014cispolysys}
Milan. Korda, Didier. Henrion, and Colin~N. Jones.
\newblock Convex computation of the maximum controlled invariant set for
  polynomial control systems.
\newblock {\em SIAM Journal on Control and Optimization}, 52(5):2944--2969,
  2014.

\bibitem[Lau04]{laub2004matrixanalysis}
Alan~J. Laub.
\newblock {\em Matrix Analysis For Scientists And Engineers}.
\newblock Society for Industrial and Applied Mathematics, USA, 2004.

\bibitem[LO21]{zexiang2021adhs}
Zexiang Liu and Necmiye Ozay.
\newblock Safe online planning in unknown nonconvex environments with implicit
  controlled invariant sets.
\newblock {\em IFAC-PapersOnLine}, 2021.

\bibitem[LO22]{liu22conv}
Zexiang Liu and Necmiye Ozay.
\newblock On the convergence of the backward reachable sets of robust
  controlled invariant sets for discrete-time linear systems, 2022.
\newblock accepted to CDC 2022.

\bibitem[Lof04]{lofberg2004yalmip}
Johan Lofberg.
\newblock Yalmip: A toolbox for modeling and optimization in matlab.
\newblock In {\em 2004 IEEE international conference on robotics and automation
  (IEEE Cat. No. 04CH37508)}, pages 284--289. IEEE, 2004.

\bibitem[LS15]{lazar2015finstepmpc}
Mircea Lazar and Veaceslav Spinu.
\newblock Finite-step terminal ingredients for stabilizing model predictive
  control.
\newblock {\em IFAC-PapersOnLine}, 48(23):9--15, 2015.
\newblock 5th IFAC Conference on Nonlinear Model Predictive Control NMPC 2015.

\bibitem[LTJ18]{legat2018cishybsys}
Beno{\^{\i}}t Legat, Paulo Tabuada, and Rapha{\"{e}}l~M. Jungers.
\newblock Computing controlled invariant sets for hybrid systems with
  applications to model-predictive control.
\newblock In {\em 6th {IFAC} Conference on Analysis and Design of Hybrid
  Systems, {ADHS} 2018, Oxford, UK}, pages 193--198, 2018.

\bibitem[MD13]{mueller2013mpcquadcopter}
M.~W. {Mueller} and R.~{D'Andrea}.
\newblock A model predictive controller for quadrocopter state interception.
\newblock In {\em 2013 European Control Conference (ECC)}, pages 1383--1389,
  July 2013.

\bibitem[MHO18]{munir2018highdegreelyap}
Sarmad Munir, Morten Hovd, and Sorin Olaru.
\newblock Low complexity constrained control using higher degree lyapunov
  functions.
\newblock {\em Automatica}, 98:215 -- 222, 2018.

\bibitem[MSR05]{mayne2005robustmpc}
D.Q. Mayne, M.M. Seron, and S.V. Raković.
\newblock Robust model predictive control of constrained linear systems with
  bounded disturbances.
\newblock {\em Automatica}, 41(2):219--224, 2005.

\bibitem[OTH19]{oustry2019innerapproxcis}
A.~{Oustry}, M.~{Tacchi}, and D.~{Henrion}.
\newblock Inner approximations of the maximal positively invariant set for
  polynomial dynamical systems.
\newblock {\em IEEE Control Systems Letters}, 3(3):733--738, July 2019.

\bibitem[PAT21]{pannocchi2021iros}
Luigi Pannocchi, Tzanis Anevlavis, and Paulo Tabuada.
\newblock Trust your supervisor: quadrotor obstacle avoidance using controlled
  invariant sets.
\newblock In {\em 2021 IEEE/RSJ International Conference on Intelligent Robots
  and Systems (IROS)}, pages 9219--9224, 2021.

\bibitem[RKKM05]{rakovic2005invariant}
Sasa~V Rakovic, Eric~C Kerrigan, Konstantinos~I Kouramas, and David~Q Mayne.
\newblock Invariant approximations of the minimal robust positively invariant
  set.
\newblock {\em IEEE Transactions on automatic control}, 50(3), 2005.

\bibitem[RT17]{rungger2017rcislinsys}
M.~{Rungger} and P.~{Tabuada}.
\newblock Computing robust controlled invariant sets of linear systems.
\newblock {\em IEEE Transactions on Automatic Control}, 62(7):3665--3670, July
  2017.

\bibitem[TJ15]{tahir2015lowcomplnormbound}
F.~{Tahir} and I.~M. {Jaimoukha}.
\newblock Low-complexity polytopic invariant sets for linear systems subject to
  norm-bounded uncertainty.
\newblock {\em IEEE Transactions on Automatic Control}, 60(5):1416--1421, 2015.

\bibitem[WO20]{wintenberg2020implicit}
Andrew Wintenberg and Necmiye Ozay.
\newblock Implicit invariant sets for high-dimensional switched affine systems.
\newblock In {\em 2020 59th IEEE Conference on Decision and Control (CDC)},
  pages 3291--3297. IEEE, 2020.

\bibitem[ZS14]{zhou2014flatquadrotors}
D.~Zhou and M.~Schwager.
\newblock Vector field following for quadrotors using differential flatness.
\newblock {\em 2014 IEEE International Conference on Robotics and Automation
  (ICRA)}, pages 6567--6572, 2014.

\end{thebibliography}


\end{document}